\documentclass[11pt]{article}
\usepackage[T1]{fontenc}
\usepackage{amssymb,amsmath,amsthm}
\usepackage{graphicx}
\usepackage{indentfirst}
\usepackage[numbers,sort&compress]{natbib}
\usepackage{authblk}
\usepackage{float}
\usepackage{latexsym,amsfonts,graphics,mathtools}
\usepackage{abstract}
\usepackage{siunitx}
\usepackage{geometry}
\usepackage{overpic}
\usepackage[french,english]{babel}
\usepackage{color}
\usepackage[colorlinks,linkcolor=blue]{hyperref}
\usepackage{tikz}
\usepackage{subcaption}
\usepackage{float}

\usetikzlibrary{decorations.pathreplacing}
\usetikzlibrary{positioning} 
\usepackage{xcolor}
\usepackage{environ}
\usepackage{enumitem}
\usepackage{parskip}
\newlist{encase}{enumerate}{1}
\setlist[encase]{label=Case \arabic*, leftmargin=2cm}

\numberwithin{equation}{section}
\newtheorem{theorem}{Theorem}[section]
\newtheorem{lemma}[theorem]{Lemma}

\newtheorem{proposition}[theorem]{Proposition}
\newtheorem{definition}[theorem]{Definition}
\newtheorem{remark}{Remark}

\makeatletter
\renewenvironment{proof}[1][\proofname]{\par
  \pushQED{\qed}%
  \normalfont \topsep6\p@\@plus6\p@\relax
  \trivlist
  \item[\hskip\labelsep
        \bfseries
    #1\@addpunct{.}]\ignorespaces
}{%
  \popQED\endtrivlist\@endpefalse
}
\makeatother

\makeatletter
\newsavebox{\measure@tikzpicture}
\NewEnviron{scaletikzpicturetowidth}[1]{%
  \def\tikz@width{#1}%
  \def\tikzscale{1}\begin{lrbox}{\measure@tikzpicture}%
  \BODY
  \end{lrbox}%
  \pgfmathparse{#1/\wd\measure@tikzpicture}%
  \edef\tikzscale{\pgfmathresult}%
  \BODY
}
\makeatother

\title{Multivariate multifractal analysis of  L\'evy functions. Part II: Validity of the multifractal formalism}
\author[1]{St\'ephane Jaffard \thanks{stephane.jaffard@u-pec.fr}}
\author[2]{Lingmin Liao \thanks{lmliao@whu.edu.cn}}
\author[1]{Qian Zhang \thanks{qian.zhang@u-pec.fr}}
\affil[1]{{\small{Université Paris Est Créteil, Université Gustave Eiffel, CNRS, LAMA UMR8050, F-94010 Créteil, France}}}
\affil[2]{{\small{School of Mathematics and Statistics, Wuhan University, Wuhan, China}}}

\begin{document}
\maketitle
\begin{abstract}
In this article, we determine the multivariate multifractal Legendre spectra of shifted L\'evy functions.
This allows us to explore how the validity of the multivariate multifractal formalism depends on the shift parameter. This article is a continuation of \cite{jaffardliaoqian} where the corresponding bivariate multifractal spectra were explored. 
\end{abstract}

\section{Introduction}
This is the second part of 2-parts papers concerning the bivariate multifractal analysis of shifted L\'evy functions. 
The discussion begins by recalling the definition of these functions.  Consider first the $1$-periodic 'saw-tooth' function defined by 
\begin{equation}\nonumber
\begin{aligned}
\{x\}=
\left\{
             \begin{array}{lr}
            x-\lfloor x \rfloor - \frac{1}{2},&\quad \text{if}\ x\ \notin\ \mathbb{Z},\\
            0,&\quad \text{if}\ x\ \in\ \mathbb{Z},
             \end{array}
\right.
\end{aligned}
\end{equation}
where $\lfloor x \rfloor$ denotes the integer part of the real number $x$.

\begin{definition}
Let  $\alpha > 0$ and $b \in \mathbb{N}$ with $ b \geqslant 2$;  the {L\'evy function}  $L_\alpha^b$ is defined by 
\begin{equation} \label{lef}  \forall x \in \mathbb{R}, \qquad 
L_\alpha^b(x)=\sum_{i=1}^{\infty}\frac{\{b^ix\}}{b^{\alpha i}} .
\end{equation}
\end{definition}

In Part I, we studied the bivariate multifractal spectra of shifted L\'evy functions \cite{jaffardliaoqian}, and in this second part, we determine their bivariate Legendre spectra. Comparing the two  will  improve our understanding of the domain of validity of the bivariate multifractal formalism.

The successes of multifractal analysis have been primarily limited to the examination of individual signals or images.  Considering the increasing availability of collections of related signals for numerous contemporary real-world applications, the restriction to univariate analysis represents a significant drawback. The study of multivariate multifractal analysis aims to establish a theoretical foundation by investigating the properties and limitations of the natural extension of univariate formalism to multivariate formalism.  

 In practice, one often analyzes with related signals for which a time shift of unknown value is present. Typical examples are supplied by EEG signals recorded at several locations in the brain:  these signals  encapsulate  information issued from different parts of the brain which are delayed by an unknown quantity \cite{Frontiers2020}. In such situations, it is important to understand how this time shift affects the joint multifractal properties of the two signals. Our purpose in this paper is to investigate this question on a particular toy model: we will consider shifted L\'evy functions  and study the dependency of the bivariate spectra as a function of  the shift.

Multivariate multifractal analysis is a simultaneous multifractal analysis of several pointwise regularity exponents derived from one or several functions or sample paths of stochastic processes. The fundamental definitions introduced in Part I are briefly recalled below. Suppose that the $d$ functions $f_i(1\leqslant i\leqslant d)$ defined on $\mathbb{R}$ are associated with pointwise regularity exponents $h_{f_i}(x)$ respectively. Given $H=(H_1, H_2,\cdots,H_n)\in \mathbb{R}^d$, one is  interested in the level sets
\begin{equation} \nonumber
\begin{aligned}
E_{f_1,f_2,\cdots,f_d}(H)=\ \big\{x \in \mathbb{R}^d: h_{f_1}(x) = H_1, h_{f_2}(x) = H_2,\cdots,h_{f_d}(x) = H_d\big\}.
\end{aligned}
\end{equation}
The multivariate multifractal spectrum is a (multivariate) function defined as
\begin{equation} \nonumber
\mathcal{D}_{f_1,f_2,\cdots,f_d}: H\longmapsto \dim_H\left(E_{f_1,f_2,\cdots,f_d}(H)\right),
\end{equation}
where $\dim_H$ stands for the Hausdorff dimension.
The support of the multivariate multifractal spectrum consists of the set of the vectors $H$, such that $E_{f_1,f_2,\cdots,f_d}(H)  \neq \emptyset$.

The multifractal Legendre spectrum is now extended to the multivariate setting. The main distinction from previous articles on the subject lies in the choice of multiresolution quantities: oscillations of order $n$ are employed, and $d$ functions $f_i(1\leqslant i\leqslant d)$ are considered,  each associated with its respective oscillation $d_{\lambda}^{n,(i)}$. To define the oscillations, the notation of finite differences  of order $M$ is recalled, following the inductive construction in \cite[p23]{clausel2008quelques}.
 
 \begin{definition}
Let $f$ be a locally bounded function. Let $h\in \mathbb{R}$, $x\in \mathbb{R}$, and $M\geqslant 1$. Then
 \begin{equation}
    \nonumber
    \Delta_f^1(x,h)=f(x+h)-f(x),
\end{equation}
and, if $M\geqslant 2$, then
\begin{equation}
\begin{aligned}
     \nonumber
    \Delta_f^M(x,h)=\Delta_f^{M-1}(x+h,h)-\Delta_f^{M-1}(x,h).
\end{aligned}
\end{equation}
\end{definition}

Let $j\in\mathbb{Z}$ and $k \in \mathbb{Z}$. Denote by $\lambda(=\lambda(j,k))$ the $b$-adic interval $\left[\frac{k}{b^j},\frac{k+1}{b^j}\right)$ and $3\lambda$ the interval of the same center and three times wider. For clarity, it is assumed throughout that the fractions of the form $\frac{k}{b^j}(\forall j)$ are in reduced form , i.e. $\gcd(k,b)=1,$ ($k\wedge b=1)$. The behavior of L\'evy functions is primarily governed by the scale parameter $ j $, while the specific value of $k $ plays no essential role. In the text, the rational number $\frac{k}{b^j} $ is used to represent a generic point at scale $ j $, where $k$ is any admissible integer satisfying the coprimality condition. For $x\in\mathbb{R}$, let $\lambda_j(x)$ be the unique $b$-adic interval of width $b^{-j}$ which contains $x$. The { \em $M$-th order oscillation}  of $f$ on $3 \lambda$ is
\begin{equation}
   \nonumber
d_{\lambda}^M=\sup\limits_{(x,h):\ x,\ x+Mh \ \in 3\lambda}|\Delta_f^M(x,h)|.
\end{equation}
For all $1\leqslant i\leqslant d$, the oscillations $d_{\lambda}^{M,(i)}$ corresponding to the functions $f_i$ are defined on the same ($b$-adic) grid. 


In all generality, in order to perform the multifractal analysis of smooth functions, the definition of the oscillation has to be replaced by   differences of order larger than 1; however, in the case of L\'evy functions,  the derived Legendre spectra  do not depend on the value of $M \geqslant 2$, even if the H\"older exponent takes values larger than 2,  see \cite{qianjump}. This explains why, in the following, the index $M$ in the notation of oscillation will often be dropped, in which case it is implicitly assumed that $M=2$. Moreover, as proved in \cite[section 2.2]{qianjump}, the  Legendre spectra  based on different $b$-adic splittings do not depend on the value of $b$.  Therefore, without loss of generality, the  $b$-adic segmentations  will be used for the study of  $b$-adic L\'evy functions \eqref{lef} in the multivariate setting.

Denote by $\Lambda_j$ the collection of $b$-adic intervals of width $b^{-j}$ and by $\Lambda$ the collection of all $b$-adic intervals. Let  $d_{\lambda}^{(i)}$ denote the oscillation of $f_i$ on the interval $3\lambda$ for $1\leqslant i \leqslant d$;
the {\em multivariate multifractal structure function} is 
\begin{equation} \label{sructurefunction}
    \forall r=(r_1,r_2,\cdots,r_d)\in \mathbb{R}^d,\  S(r,j)=b^{-j}\sum\limits_{\lambda\in\Lambda_j}\prod\limits_{i=1}^d\big(d_{\lambda}^{(i)}\big)^{r_i}.
\end{equation}

The {\em scaling function}  of $(f_1, \dots , f_d) $ is defined as
\begin{equation} \label{scalingfunction}
 \zeta(r)=\liminf\limits_{j\to+\infty}\frac{\log(S(r,j))}{\log(b^{-j})}.
\end{equation}
The {\em multivariate multifractal Legendre spectrum} is
\begin{equation} \label{mml}
    \forall H\in\mathbb{R}^d, \ \mathcal{L}_{f_1,f_2,\cdots,f_d}(H)=\inf_{r\in\mathbb{R}^d}(1-\zeta(r)+H\cdot r),
    \end{equation}
where $H\cdot r$ denotes the usual scalar product in $\mathbb{R}^d$.

In the univariate setting, the  first formulation of the multifractal formalism was proposed in the seminal article of U. Frisch and G. Parisi \cite{frisch1980fully}. The  formulation of its multivariate extension, also based on increments,    was introduced by C. Meneveau et al. in \cite{meneveau1990joint}. In the realm of hydrodynamic turbulence, this extended formalism is commonly referred to as "grand canonical," drawing inspiration from thermodynamics, which serves as the foundational source for the intuitions and constructions underpinning multifractal analysis \cite{arneodo1998singularity}.  The formulation we use in the present paper is based on oscillations instead of  increments. This constitutes a major difference, especially for the estimation of negative moments (i.e. when the vector $r$   in \eqref{mml} has negative components)  where it is well known that oscillations lead to numerically stable estimations, and it is not the case for increments (see e.g. \cite{jaffard2007wavelet,Jaffard2015,lashermes2008comprehensive}  and references therein for  discussions of the instabilities that show up when estimating negative moments). 
Notably, in sharp contrast  with the univariate  case,  the key  upper bound 
\begin{equation} \label{uppbou} \forall H\in \mathbb{R}, \qquad 
    \mathcal{D}(H)\leqslant \mathcal{L}(H),  
\end{equation} does not extend in all generality to the multivariate setting, see \cite{jaffard2019multifractal} (note however that some partial generic results remain valid,  see e.g.  \cite{ben2016baire,abid2017prevalent}). The Legendre spectrum does not yield an upper bound for the bivariate multifractal spectrum in several cases;   this  drawback has been the object of several studies, see in particular \cite{jaffard2019multifractal} where conditions are derived under which the upper bound is valid;  more recently, in \cite{seuret2024multivariate}, S. Seuret investigated the bivariate multifractal analysis of pairs of Borel probability measures, proving that, in this case too,   the Legendre spectrum does not yield an upper bound for the bivariate multifractal spectrum.

The multifractal formalism has not been extensively explored in practical applications. One of the  goals of this article is to verify the numerical accuracy of this formalism on some toy examples. 
This research aims to provide concrete evidence of the robustness and reliability of multifractal approaches based on oscillations when applied to real data.

Our main purpose in this paper is to contribute to the  comprehension of this question in the following way: the bivariate multifractal formalism is studied in the particular case of a L\'evy function \eqref{lef} and its shift by a quantity $y$. More precisely, the analysis focuses on how the result depends on the arithmetic properties of the shift $y$. It is shown that the result does not depend only on the $b$-adic approximation properties of $y$, but actually on the precise sequence of digits of $y$ of its expansion in basis $b$. Our analysis reveals that, for  almost all shift $y$, the bivariate  multifractal formalism does not hold, even in  this simplified situation involving a function and its translates. Moreover, the bivariate Legendre spectrum is bounded above by the bivariate multifractal spectrum. It highlights an unexpected inverse dependency, marks a fundamental divergence from the univariate framework where \eqref{uppbou}  holds. This limitation of the validity of the multifaractal formalism has critical implications for practical applications, notably in non-stationary signal processing or complex systems modeling, where the bivariate interpretation of multifractal spectra is often essential. 

Our paper is organized as follows. In Section \ref{section2}, the main results of our paper are presented, which concern the determination of the multivariate multifractal Legendre spectra of shifted L\'evy functions; in particular, in view of the results collected in Part I, this allows  to discuss the validity or invalidity of the multifractal formalism depending on the value taken by the shift parameter $y$. Then the proofs concerning the multivariate Legendre spectra are supplied in Section \ref{section5}. Finally, in the conclusion (Section \ref{sec:concl}), the numerical results supplied by the respectively wavelet-based and oscillation-based multifractal formalisms are compared. 

\section{Statement of  the main results} \label{section2}

To state results on the multivariate multifractal Legendre spectra of shifted L\'evy functions, it is necessary introduce some notations.
 Let $\mathcal{A}=\{0,1,\cdots,b-1\}$. Every number $x$ in $[0,1)$ admits a $b$-ary expansion
\begin{equation} \label{numx} 
x=\sum\limits_{i=1}^{\infty} \frac{\varepsilon_i}{b^i},\quad\forall \varepsilon_i\ \in\mathcal{A}
\end{equation}
which is unique except when $x$ is not a $b$-adic rational number, in which case two distinct expansions exist. 
Conversely, to an  infinite word $\varepsilon=\varepsilon_1\varepsilon_2\cdots$, is associated the real number $x$ given by \eqref{numx} through  the projection map $\Pi:\mathcal{A}^{\mathbb{N}}\longmapsto[0,1)$   defined by 
\begin{equation}\nonumber
\Pi(\varepsilon)=\sum\limits_{i=1}^{\infty} \frac{\varepsilon_i}{b^i},\quad \forall \varepsilon=\varepsilon_1\varepsilon_2\cdots\in\mathcal{A}^{\mathbb{N}}.
\end{equation}
This map  is one-to-one except for the $b$-adic rationals (where it is two-to-one). Let $y=\frac{k}{b^m}$ be a $b$-adic rational number, $(\text{with}\ \gcd(k,b)=1, \ \text{which will be denoted as}\ k\wedge b=1)$. Then the $b$-ary expansion $\varepsilon_1\varepsilon_2\cdots$ of $\frac{k}{b^m}$ either terminates with zeros or becomes eventually constant with digit $(b-1)$'s. More precisely, it falls into one of the following two cases:
\begin{itemize}
\item $\forall k>m; \quad \varepsilon_m= a,\ \varepsilon_k=0,$ 
\item
$\forall k>m,  \quad \varepsilon_m = a-1,\ \varepsilon_k=b-1. $
\end{itemize}
where $a\in \mathcal{A}\backslash\{0\}$.

Throughout this paper, when referring to the $b$-ary expansion of the $b$-adic rational number $\frac{k}{b^m}$, we will, unless otherwise specified, adopt the convention that the expansion satisfies $\varepsilon_m=a$, $\varepsilon_k=0,\forall k>m$. The alternative representation, with $\varepsilon_m=a-1$, $\varepsilon_k=b-1,\forall k>m$, leads to the same conclusions and can be treated similarly.

 \bigskip
Define the shifted L\'evy function by $L_{\alpha}^{b,y}(x)=L_{\alpha}^{b}(x-y)$. The corresponding bivariate multifractal spectra $\mathcal{D}_{L_{\alpha_1}^b, L_{\alpha_2}^{b,y}}$ were estimated in \cite{jaffardliaoqian}. 

Note that  $\{ x \} = -\{ 1-x \} $, implying that  $L_{\alpha}^b (x) = -L_{\alpha}^b (1-x) $ and the operation $x \rightarrow 1-x$ changes the digit $\varepsilon_i $ in the $b$-ary expansion of $x$ and $1-x$ as follows: $\varepsilon_i $ in the $b$-ary expansion of $x$ becomes $\varepsilon_{b-i} $ in the $b$-ary expansion $1-x$. Thus, replacing the shift $y$ by $1-y$ does not affect the bivariate multifractal spectrum. 

The bivariate multifractal spectrum of the L\'evy function $L_{\alpha_1}^{b}$ and the translated L\'evy function $L_{\alpha_2}^{b,y},$ was given in Part I, see  \cite{jaffardliaoqian}; we recall some of these results.

\begin{theorem} \cite[Theorem 2.1]{jaffardliaoqian}\label{hausdorffdimension} Let $\alpha_1, \alpha_2>0$. For almost every $y$, the bivariate multifractal spectrum $\mathcal{D}_{L_{\alpha_1}^b,L_{\alpha_2}^{b,y}}$
     of  the L\'evy functions $L_{\alpha_1}^{b}$ and $L_{\alpha_2}^{b,y}$,  is given by
$$
\mathcal{D}_{L_{\alpha_1}^b,L_{\alpha_2}^{b,y}}(H_1, H_2)=
  \begin{cases}
    \min\left\{\frac{H_1}{\alpha_1}, \frac{H_2}{\alpha_2}\right\},       & \quad {\rm if}\ (H_1 , H_2)\in  [0,\alpha_1]\times \big[0,\alpha_2],\\
-\infty       & \quad {\rm else}.
  \end{cases}
$$ 
\end{theorem}

\begin{proposition}\cite[Proposition 2.4]{jaffardliaoqian}\label{examplebadicrational}
Let $\alpha_1, \alpha_2>0$. For any $y \in \mathbb{Q} $, the following dichotomy holds for the bivariate multifractal spectrum $\mathcal{D}_{L_{\alpha_1}^b,L_{\alpha_2}^{b,y}}$
     of  the L\'evy function $L_{\alpha_1}^{b}$ and the translated L\'evy function $L_{\alpha_2}^{b,y}$:
\begin{itemize}
    \item 
if $y$ is a $b$-adic rational number, then the bivariate spectrum takes values on the diagonal defined by $\frac{H_1}{{\alpha_1}} = \frac{H_2}{{\alpha_2}}$, and  
\begin{equation}
    \begin{aligned}
        \nonumber
        \mathcal{D}_{L_{{\alpha_1}}^b,L_{\alpha_2}^{b,y}}(H_1, H_2)=
  \begin{cases}
    \frac{H_1}{{\alpha_1}},       & \quad {\rm if}\ H_2 = \frac{{\alpha_2}}{{\alpha_1}}H_1,\ \text{and}\ H_1\in [0,\alpha_1],\\
-\infty       & \quad {\rm else}.
  \end{cases}
    \end{aligned}
\end{equation}
\item If $y$ is a non-$b$-adic rational number, then the bivariate spectrum is 
$$
\mathcal{D}_{L_{{\alpha_1}}^b,L_{{\alpha_2}}^{b,y}}(H_1, H_2)=
  \begin{cases}
    \min\left\{\frac{H_1}{{\alpha_1}}, \frac{H_2}{{\alpha_2}}\right\},       & \quad {\rm if}\ (H_1 , H_2)\in  [0,{\alpha_1}]\times \big[0,{\alpha_2}],\\
-\infty       & \quad {\rm else}.
  \end{cases}
$$
\end{itemize}
\end{proposition}

  For L\'evy functions,  Legendre spectra based on oscillations of any order $M \geqslant 2$ yield the same result, even for large values of $\alpha$, see \cite[Proposition 2.5]{qianjump}. 
  As a consequence, here and thereafter, we assume that the order of oscillations is $M=2$.
  
  To state our main results, we now define a digit-block sequence associated with the $b$-ary expansion of $y$, which is crucial for classifying the validity of the multifractal formalism. 
Let $u_1=1$. Suppose that $\varepsilon_1$ is given in the $b$-ary expansion of $y$. Define recursively for all $k \geqslant 2$:  

\begin{equation}
\begin{aligned}\label{rst}
    &u_{k}=\inf\left\{i:\ i\geqslant u_{k-1},\ \text{and}\ \varepsilon_{i} \notin \mathcal{B}(u_{k-1})\right\},
\end{aligned}  
\end{equation}
where $\mathcal{B}(u_{k-1})$ represents the block containing index $u_{k-1}$, defined as follows:
\begin{itemize}
    \item If $\varepsilon_{u_{k-1}} = 0$, then $\mathcal{B}(u_{k-1})$ is the longest contiguous block of zeros starting at $u_{k-1}$.
    \item If $\varepsilon_{u_{k-1}} = b-1$, then $\mathcal{B}(u_{k-1})$ is the longest contiguous block of $(b-1)$'s starting at $u_{k-1}$.
    \item Otherwise, $\mathcal{B}(u_{k-1})$ is the longest contiguous block of digits from $\mathcal{A} \setminus \{0, b-1\}$ starting at $u_{k-1}$.
\end{itemize}

For all $k \geqslant 1$,  the length of the $k$-th block as  
\begin{equation}
    \label{w_k}
    w_k = u_{k+1} - u_k.
\end{equation}  
Clearly, $\{w_k\}$ forms a sequence of positive integers.

The following theorem gives the bivariate multifractal Legendre spectrum of $L_{\alpha_1}^{b}$ and $L_{\alpha_2}^{b,y}$.



\begin{theorem}\label{multitheofinal}
       Let $y \in [0, 1]$. If y is a $b$-adic rational number, then $\{w_k\}$ as defined in \eqref{w_k} is a finite sequence, and the bivariate multifractal Legendre spectrum of $L_{\alpha_1}^{b}$ and $L_{\alpha_2}^{b,y}$ based on oscillations is 
        \begin{equation}
        \nonumber
\mathcal{L}_{L_{\alpha_1}^b,L_{\alpha_2}^{b,y}}(H_1, H_2)= \left\{
             \begin{array}{ll}
           \frac{H_1}{\alpha_1}=\frac{H_2}{\alpha_2},&\quad {\rm for} \quad \big(H_1, \frac{\alpha_2}{\alpha_1} H_1\big) \quad {\rm with} \quad \ H_1\in [0,  \alpha_1 ],\vspace{1ex}\\
            -\infty,&\quad {\rm else}.
             \end{array}
\right.
    \end{equation}
 If $y$ is not a $b$-adic rational number, then the following dichotomy holds: 

\begin{itemize}
    \item 
If $\{w_k\}$ is a bounded infinite sequence, then the bivariate multifractal Legendre spectrum based on the oscillations is 
\begin{equation}
        \nonumber
\mathcal{L}_{L_{\alpha_1}^b,L_{\alpha_2}^{b,y}}(H_1, H_2)= \left\{
             \begin{array}{ll}
           \frac{H_1}{\alpha_1}+\frac{H_2}{\alpha_2}-1,&\quad {\rm if}\ (H_1,H_2)\ \in\ [0,\alpha_1]\times [0,\alpha_2],\ \\&\quad \text{\rm and}\ \frac{H_1}{\alpha_1}+\frac{H_2}{\alpha_2}-1\geqslant 0,\vspace{1ex}\\
            -\infty,&\quad {\rm else}.
             \end{array}
\right.
    \end{equation}
 \item 
If $\{w_k\}$ is an unbounded infinite sequence, then the bivariate multifractal Legendre spectrum based on the oscillations is
\begin{equation}
        \nonumber
\mathcal{L}_{L_{\alpha_1}^b,L_{\alpha_2}^{b,y}}(H_1, H_2)= \begin{cases}
    \min\left\{\frac{H_1}{\alpha_1}, \frac{H_2}{\alpha_2}\right\},       & \quad {\rm if}\ (H_1,H_2) \in \left[0,\alpha_1\right]\times\left[0,{\alpha_2}\right],\\
-\infty       &\quad {\rm else}.
  \end{cases}
    \end{equation}
    \end{itemize}
 \end{theorem}
 The proof of Theorem \ref{multitheofinal} will be provided in Section \ref{section5.1}.
 \begin{remark}
    For almost every $y$, the bivariate Legendre spectrum is bounded above by the bivariate multifractal spectrum.
 \end{remark}

\begin{remark}
   If $y$ is a $b$-adic rational number, it follows from Proposition \ref{examplebadicrational} and Theorem \ref{multitheofinal} that the bivariate multifractal spectrum coincides with the bivariate multifractal Legendre spectrum based on the oscillations. If $\{w_k\}$ is an unbounded sequence, then by Theorem \ref{multitheofinal} and \cite[Theorems 2.1 and 2.2]{jaffardliaoqian},
   \begin{equation}
    \nonumber
    \mathcal{D}_{L_{\alpha_1}^b,L_{\alpha_2}^{b,y}}(H_1, H_2)\leqslant \mathcal{L}_{L_{\alpha_1}^b,L_{\alpha_2}^{b,y}}(H_1, H_2).
\end{equation}
But if $\{w_k\}$ is a bounded sequence, then the bivariate multifractal spectrum is a subcase of \cite[Theorem 2.3]{jaffardliaoqian}. Using Theorem \ref{multitheofinal}, it holds that
 \begin{equation}
    \label{LlesthanD}
    \mathcal{L}_{L_{\alpha_1}^b,L_{\alpha_2}^{b,y}}(H_1, H_2)\leqslant \mathcal{D}_{L_{\alpha_1}^b,L_{\alpha_2}^{b,y}}(H_1, H_2).
\end{equation}
Furthermore, combining \cite[Remark 1]{jaffardliaoqian}, it turns out that \eqref{LlesthanD} holds for almost every $y$. This analysis yields new counterexamples, effectively demonstrating that the bivariate multifractal Legendre spectra fail to establish an upper bound on the bivariate multifractal spectra, even in the very special setting of a function and its translates. Instead, it reveals an unexpected inverse correlation, diverging from conventional expectations. This phenomenon contrasts sharply with the univariate setting, where the upper bound  \eqref{uppbou}   always holds.
\end{remark}




Next, we give the multivariate multifractal spectrum for $n$ L\'evy functions. Recall that  $w_k$  is defined in \eqref{rst}; denote 
\begin{equation}
    \begin{aligned}
        \label{s123}
        &\mathcal{S}_1=\{y:\{w_k\}\ \text{is a bounded sequence}\},\\
        &\mathcal{S}_2=\{y:\{w_k\}\ \text{is an unbounded sequence}\},\\
        &\mathcal{S}_3=\{y: y\ \text{is a $b$-adic rational number}\}.
    \end{aligned}
\end{equation}
Let $n\geqslant 2$ and $\alpha_i>0$, $y_i\in[0,1)$ for $i=1,\cdots, n$. If $y_1=0$, and $y_i$ is $b$-adic rational for any $2\leqslant i\leqslant n$, then the  multivariate multifractal Legendre spectrum for $n$ L\'evy functions $L_{\alpha_i}^{b,-y_i}$ ($1\leqslant i\leqslant n$) is as follows.
\begin{theorem}
    \label{multivariatemultifractalLegendrespectrum_nfunctionsrational}
Let $L_{\alpha_i}^{b,-y_i}$ ($1\leqslant i\leqslant n$) be the $i$-th L\'evy function with shift $-y_i$. If $y_1=0$, and $y_i\in \mathcal{S}_3,(2\leqslant i \leqslant n)$, then the multivariate multifractal Legendre spectrum of these $n$ functions based on the oscillations is
\begin{small}
\begin{equation}
    \begin{aligned}
        \nonumber
    \mathcal{L}_{L_{\alpha_1}^{b,-y_1},L_{\alpha_2}^{b,-y_2},\cdots,L_{\alpha_n}^{b,-y_n}}(H_1, H_2,\cdots,H_n)=
    \left\{
             \begin{array}{ll}
           \frac{H_1}{\alpha_1},\ &\text{\rm if}\ (H_1,\cdots,H_n)\in [0,\alpha_1]\times\cdots\times[0,\alpha_n],\vspace{1ex}\\&H_i=\frac{\alpha_i}{\alpha_1}H_1,\ \text{\rm for}\ 2\leqslant i \leqslant n,\ \vspace{1ex}\\
            -\infty,&\text{\rm else}.
             \end{array}
\right.
    \end{aligned}
\end{equation}
\end{small}
\end{theorem}
Theorem \ref{multivariatemultifractalLegendrespectrum_nfunctionsrational} will be proven in Section \ref{section5.2}.

If $y_1=0$, and $y_i(2\leqslant i\leqslant n)$ is not $b$-adic rational satisfying that the sequence $\{w_k\}$ defined in \eqref{w_k} is bounded in the $b$-ary expansion of $y_i$ for $2\leqslant i\leqslant m$, and the sequence $\{w_k\}$ is unbounded in the $b$-ary expansion of $y_i$ for $m+1\leqslant i\leqslant n$, then the  multivariate multifractal Legendre spectrum for $n$ L\'evy functions $L_{\alpha_i}^{b,-y_i}$ ($1\leqslant i\leqslant n$) is given by the following theorem.
\begin{theorem}
    \label{multivariatemultifractalLegendrespectrum_nfunctions}
Let $L_{\alpha_i}^{b,-y_i}$ ($1\leqslant i\leqslant n$) be the $i$-th L\'evy function with shift $-y_i$. If $y_1=0, y_i\in \mathcal{S}_3,(2\leqslant i\leqslant m),$ and $y_i\in \mathcal{S}_1,(m+1\leqslant i \leqslant n)$, then the multivariate multifractal Legendre spectrum of these $n$ functions based on the oscillations is given by
\begin{small}
\begin{equation}
    \begin{aligned}
        \nonumber
\mathcal{L}_{L_{\alpha_1}^{b,-y_1},L_{\alpha_2}^{b,-y_2},\cdots,L_{\alpha_n}^{b,-y_n}}(H_1, H_2,\cdots,H_n)=
    \left\{
             \begin{array}{ll}
           \frac{H_1}{\alpha_1}+\frac{H_{m+1}}{\alpha_{m+1}}-1,\ &\text{\rm if}\ (H_1,\cdots,H_n)\in [0,\alpha_1]\times\cdots\times[0,\alpha_n],\vspace{1ex}\\&H_i=\frac{\alpha_i}{\alpha_1}H_1,\ \text{\rm for}\ 2\leqslant i \leqslant m,\ \vspace{1ex} \\&H_i=\frac{\alpha_i}{\alpha_{m+1}}H_{m+1}, \ \text{\rm for}\ m+2\leqslant i \leqslant n,\ \vspace{1ex}\\&\text{\rm and}\ \frac{H_l}{\alpha_{1}}+\frac{H_{m+1}}{\alpha_{m+1}}-1\geqslant 0,\vspace{1ex}\\
            -\infty,&\text{\rm else}.
             \end{array}
\right.
    \end{aligned}
\end{equation}
\end{small}
\end{theorem}
 Theorem \ref{multivariatemultifractalLegendrespectrum_nfunctions} will be proven in Section \ref{section5.3}.

Let us specify this result in the case $b=2$ and $y=\frac{1}{3}$ or $y=\frac{1}{5}$, then the bivariate multifractal Legendre spectrum is
\begin{equation}
        \label{1/3}
\mathcal{L}_{L_{\alpha_1}^2,L_{\alpha_2}^{2,y}}(H_1, H_2)= \left\{
             \begin{array}{ll}
           \frac{H_1}{\alpha_1}+\frac{H_2}{\alpha_2}-1,&\quad \text{if}\ H\ \in\ [0,\alpha_1]\times [0,\alpha_2],\ \text{and}\ \frac{H_1}{\alpha_1}+\frac{H_2}{\alpha_2}-1\geqslant 0,\vspace{1ex}\\
            -\infty,&\quad \text{\rm else}.
             \end{array}
\right.
    \end{equation}
If $b=2$ and $y=\frac{1}{2}$ or $y=\frac{1}{4}$, then the bivariate multifractal Legendre spectrum is
 \begin{equation}
        \label{1/2}
\mathcal{L}_{L_{\alpha_1}^b,L_{\alpha_2}^{b,y}}(H_1, H_2)= \left\{
             \begin{array}{ll}
           \frac{H_1}{\alpha_1}=\frac{H_2}{\alpha_2},&\quad {\rm for} \quad \big(H_1, \frac{\alpha_2}{\alpha_1} H_1\big) \quad {\rm with} \quad \ H_1\in [0,  \alpha_1 ],\vspace{1ex}\\
            -\infty,&\quad \text{\rm else}.
             \end{array}
\right.
    \end{equation}

\section{Multivariate multifractal formalism} \label{section5}
In this section, we provide the proofs of Theorems \ref{multitheofinal}, \ref{multivariatemultifractalLegendrespectrum_nfunctionsrational}, and \ref{multivariatemultifractalLegendrespectrum_nfunctions}. The general approach relies on a detailed analysis of the scaling behavior of oscillations over $b$-adic intervals. Here and thereafter, for simplicity, the notation $x\sim y$ represents the relation  $c_1y\leqslant x\leqslant c_2y$, where $c_1$ and $c_2$ are positive constants.
\subsection{Proof of Theorem \ref{multitheofinal}} \label{section5.1}
 The oscillations $d_{\lambda}^{(1)} $and $ d_{\lambda}^{(2)} $associated with the functions $L_{\alpha_1}^b$and $ L_{\alpha_1}^{b,y} $, respectively, are characterized within the interval $ 3\lambda $, with their behavior governed by the $ b $-ary expansion of the shift $y $. These oscillations serve as the basis for computing the structure function and the scaling function, from which the Legendre spectrum are subsequently obtained via the application of the Legendre transform, following their formal definitions.

The following proposition is  crucial for determining the Legendre spectrum. Using the second order (or higher) oscillation, we characterized the specific oscillation of L\'evy function $L_{\alpha}^b$ in \cite[Proposition 2.9]{qianjump}.   
\begin{proposition} \label{Propsecondorderoscillation}
   Let $\lambda(=\lambda(j,k))$ denote the $b$-adic interval $\left[\frac{k}{b^j},\frac{k+1}{b^j}\right)$ and $3\lambda$ be the interval of the same center and three times wider. Let $l$ be the smallest integer such that $\frac{k}{b^l}\in 3\lambda$, that is, $l=\min\left\{l:\frac{k}{b^l}\in 3\lambda,\  k\wedge b=1\right\}$. The  second order oscillation of $L_{\alpha}^b$ over $3\lambda$ satisfies $$d_{\lambda}\sim \Delta\left(\frac{k}{b^{l}}\right)\sim b^{-\alpha l}$$.
\end{proposition}

The following lemmas describe the oscillations of L\'evy functions $L_{\alpha_1}^b$ and $L_{\alpha_2}^{b,y}$, depending on whether $y$ is a $b$-adic rational number or not. 

\begin{lemma} \label{bi_oscillation_rational}
Let  $\lambda= \left[\frac{k'}{b^j},\frac{k'+1}{b^j}\right)$ be the $b$-adic interval and $y=\frac{k}{b^m}$ ($k\wedge b=1$) be a $b$-adic rational number, where the scale $j$ of the $b$-adic interval $\lambda$ is much larger than $m$. If the oscillation $d_{\lambda}^{(2)}$ in $3\lambda$ of the translated L\'evy function $L_{\alpha_2}^{b,y}$ is given by $d_{\lambda}^{(2)}\sim b^{-\alpha_2 i}$, then the oscillation $d_{\lambda}^{(1)}$ in $3\lambda$ of the L\'evy function $L_{\alpha_1}^b$  is
\begin{equation}
    \nonumber
    d_{\lambda}^{(1)}\sim 
    \left\{
             \begin{array}{ll}
           b^{- \alpha_1 m},&\quad \text{\rm if}\ i< m\\
            b^{- \alpha_1i'},&\quad \text{\rm if}\ i=m,\ \text{\rm and for all}\ 1\leqslant i'\leqslant m,\\
             b^{- \alpha_1i},&\quad \text{\rm if}\ i>m.
             \end{array}
\right.
\end{equation}
\end{lemma}
\begin{proof}
Denote the $b$-ary expansion of $y=\frac{k}{b^{m}}$ as $\varepsilon_1\varepsilon_2\cdots$, where $\varepsilon_{m}\neq 0$, and $\varepsilon_{t}=0$ for all $t>m$. Suppose the first $b$-adic rational number in  $3\lambda$ is $\frac{k_2}{b^{i}}$ at which $L_{\alpha_2}^{b,y}$ is discontinuous.  By Proposition \ref{Propsecondorderoscillation},  the corresponding oscillation $d_{\lambda}^{(2)}$ in  $3\lambda$ is approximately equal to the jump at $\frac{k_2}{b^{i}}$. that is, 
\begin{equation}
    \nonumber  
d_{\lambda}^{(2)}\sim\Delta\left(\frac{k_2}{b^{i}}\right)\sim b^{-\alpha_2 i}.
\end{equation}

Observe that $d_{\lambda}^{(1)}$ corresponds to the first $b$-adic rational number  in $3\lambda$ where $L_{\alpha_1}^{y}$ is discontinuous. Hence, the value of the oscillation $d_{\lambda}^{(1)}$ in $3\lambda$ of the L\'evy function $L_{\alpha_1}^{y}$  can be inferred by examining the $b$-ary expansion of $\frac{k_2}{b^{i}}+y$.
 
 When $i < m$, since $y=\frac{k}{b^m}$ is a $b$-adic rational, then  the $b$-ary expansion of $\frac{k_2}{b^{i}}+y$ satisfies $\varepsilon_{m}\neq 0$, and $\varepsilon_{t}=0$ for all $t>m$, which means it is a $b$-adic rational number, denoted $\frac{k_1}{b^m}$. As shown in Figure \ref{bianryi1<m}. Hence, the jump point of $L_{\alpha_1}^{b}$  occurs at scale $m$. By Proposition \ref{Propsecondorderoscillation},
\begin{equation}
    \nonumber
    d_{\lambda}^{(1)}\sim\Delta\left(\frac{k_1}{b^{m}}\right)\sim b^{-\alpha_1 m}.
\end{equation}
Note that $a', a^*\in \mathcal{A}\backslash \{0\}$ in the following figures. 


\begin{figure}[H]
\centering
\begin{scaletikzpicturetowidth}{\textwidth}
\begin{tikzpicture}[scale=\tikzscale]
\draw  [black](-2,0)--(7,0);
\filldraw [black] (1.5,0) circle (2pt);
\node (b) at(-1.5,0.3){\textcolor{red}{The $b$-ary expansion of $\frac{k_2}{b^{i}}$}};
\filldraw [black] (2.5,0) circle (2pt);
\filldraw [black] (5.5,0) circle (2pt);
\node[black] at(1.5,0.3){$a'$};
\node[black] at(1.7,0.3){$0$};
\node[black] at(2.1,0.3){$\cdots$};
\node[black] at(2.5,0.3){$0$};
\node[black] at(2.7,0.3){$0$};
\node[black] at(2.9,0.3){$0$};
\node[black] at(3.1,0.3){$0$};
\node[black] at(3.3,0.3){$0$};
\node[black] at(3.5,0.3){$0$};
\node[black] at(3.7,0.3){$0$};
\node[black] at(3.9,0.3){$0$};
\node[black] at(4.3,0.3){$\cdots$};
\node[black] at(4.7,0.3){$0$};
\node[black] at(4.9,0.3){$0$};
\node[black] at(5.1,0.3){$0$};
\node[black] at(5.3,0.3){$0$};
\node[black] at(5.5,0.3){$0$};
\node[black] at(5.9,0.3){$\cdots$};
\draw  [black](-2,-1)--(7,-1);
\filldraw [black] (1.5,-1) circle (2pt);
\node (b) at(-1.5,-0.7){\textcolor{red}{The $b$-ary expansion of $y$}};
\filldraw [black] (2.5,-1) circle (2pt);
\filldraw [black] (5.5,-1) circle (2pt);
\node[black] at(2.5,-0.7){$a^*$};
\node[black] at(2.7,-0.7){$0$};
\node[black] at(2.9,-0.7){$0$};
\node[black] at(3.1,-0.7){$0$};
\node[black] at(3.3,-0.7){$0$};
\node[black] at(3.5,-0.7){$0$};
\node[black] at(3.7,-0.7){$0$};
\node[black] at(3.9,-0.7){$0$};
\node[black] at(4.3,-0.7){$\cdots$};
\node[black] at(4.7,-0.7){$0$};
\node[black] at(4.9,-0.7){$0$};
\node[black] at(5.1,-0.7){$0$};
\node[black] at(5.3,-0.7){$0$};
\node[black] at(5.5,-0.7){$0$};
\node[black] at(5.9,-0.7){$\cdots$};
\node (b) at(-1.5,-1.7){\textcolor{red}{The $b$-ary expansion of $\frac{k_2}{b^{i}}+y$}};
\draw  [black](-2,-2)--(7,-2);
\filldraw [black] (1.5,-2) circle (2pt);
\node (b) at(1.5,-2.5){$\varepsilon_{i}$};
\filldraw [black] (2.5,-2) circle (2pt);
\filldraw [black] (5.5,-2) circle (2pt);
\node (c) at(2.5,-2.5){$\varepsilon_{m}$};
\node (c) at(5.5,-2.5){$\varepsilon_{j}$};
\node[black] at(2.5,-1.7){$a^*$};
\node[black] at(2.7,-1.7){$0$};
\node[black] at(2.9,-1.7){$0$};
\node[black] at(3.1,-1.7){$0$};
\node[black] at(3.3,-1.7){$0$};
\node[black] at(3.5,-1.7){$0$};
\node[black] at(3.7,-1.7){$0$};
\node[black] at(3.9,-1.7){$0$};
\node[black] at(4.3,-1.7){$\cdots$};
\node[black] at(4.7,-1.7){$0$};
\node[black] at(4.9,-1.7){$0$};
\node[black] at(5.1,-1.7){$0$};
\node[black] at(5.3,-1.7){$0$};
\node[black] at(5.5,-1.7){$0$};
\node[black] at(5.9,-1.7){$\cdots$};
\end{tikzpicture}
\end{scaletikzpicturetowidth}
\caption{} \label{bianryi1<m}
\end{figure}

When $i=m$, the $b$-ary expansion of $\frac{k_2}{b^{i}}+y$ align up to position $m$, and the presence of a carry-over depends on the digit values. As a result, there are two cases. Note that $\varepsilon_{i'}\in \mathcal{A} (\forall 1\leqslant i' \leqslant i)$, in the $b$-ary expansion of $\frac{k_2}{b^i}$. Let $\bar{a},\tilde{a}\in \mathcal{A}\backslash\{0\}$.
\begin{itemize}
    \item If $\varepsilon_m=\bar{a}$ in the $b$-ary expansion of $y$, and $\varepsilon_m=b-\bar{a}$ in the $b$-ary expansion of $\frac{k_2}{b^i}$,  then due to the occurrence of a carry in the addition process, the jump of  $L_{\alpha_1}^{y}$ could occur at all scale $i'$ satisfying $i\in [1,m-1]$. Then by Proposition \ref{Propsecondorderoscillation}, it holds that
    \begin{equation}
    \nonumber
    d_{\lambda}^{(1)}\sim\Delta\left(\frac{k_1}{b^{i'}}\right)\sim b^{-\alpha_1 i'},
\end{equation}
 where $1\leqslant i'\leqslant m-1$.
 The carry arises because the sum $\bar{a} + (b - \bar{a}) = b$, which equals the base. In $b$-ary arithmetic, this triggers a carry to the next higher digit, i.e., position $m - 1$.  Moreover, if the digit in position $m - 1$ also sums to $b$ (for example, if the digits are of the form $a$ and $b - a -1$), the carry will propagate further to position $m - 2$, and so on. This chain of carries can potentially continue up to the most significant digits. Thus, the effective position $i'$, where the \textit{first non-zero difference} appears in the expansion after accounting for the carries, can be any value in the range $1 \leqslant i' \leqslant m - 1$.  
 
    \item If $\varepsilon_m=\bar{a}$ in the $b$-ary expansion of $y$, and $\varepsilon_m =\tilde{a}$ in the $b$-ary expansion of $\frac{k_2}{b^i}$, which satisfies $ \tilde{a}\neq b-\bar{a}$, then the jump of  $L_{\alpha_1}^{y}$ occurs at scale $m$. According to Proposition \ref{Propsecondorderoscillation}, it follows that
    \begin{equation}
    \nonumber
    d_{\lambda}^{(1)}\sim\Delta\left(\frac{k_1}{b^{m}}\right)\sim b^{-\alpha_1 m}.
    \end{equation}
\end{itemize}
Thus, the oscillation $d_{\lambda}^{(1)}\sim\Delta\left(\frac{k_1}{b^{i'}}\right)\sim b^{-\alpha i'}$ with $1\leqslant i'\leqslant m$.

When $i > m$, then  the $b$-ary expansion of $\frac{k_2}{b^{i}}+y$ satisfies $\varepsilon_{i}\neq 0$, and $\varepsilon_{t}=0$ for all $t>i$, which means that it remains a 
$b$-adic rational with scale $i$, denoted $\frac{k_1}{b^i}$. As shown in Figure \ref{bianryi1>m}. Then, the jump of $L_{\alpha_1}^{b}$ is located precisely $\frac{k_1}{b^i}$. By Proposition \ref{Propsecondorderoscillation}, it holds that
\begin{equation}
    \nonumber
    d_{\lambda}^{(1)}\sim\Delta\left(\frac{k_1}{b^{i}}\right)\sim b^{-\alpha_1 i}.
\end{equation}
\begin{figure}[H]
\centering
\begin{scaletikzpicturetowidth}{\textwidth}
\begin{tikzpicture}[scale=\tikzscale]
\draw  [black](-2,0)--(7,0);
\filldraw [black] (1.5,0) circle (2pt);
\node (b) at(-1.5,0.3){\textcolor{red}{The $b$-ary expansion of $\frac{k_2}{b^{i}}$}};
\filldraw [black] (2.5,0) circle (2pt);
\filldraw [black] (5.5,0) circle (2pt);
\node[black] at(2.5,0.3){$a'$};
\node[black] at(2.7,0.3){$0$};
\node[black] at(2.9,0.3){$0$};
\node[black] at(3.1,0.3){$0$};
\node[black] at(3.3,0.3){$0$};
\node[black] at(3.5,0.3){$0$};
\node[black] at(3.7,0.3){$0$};
\node[black] at(3.9,0.3){$0$};
\node[black] at(4.3,0.3){$\cdots$};
\node[black] at(4.7,0.3){$0$};
\node[black] at(4.9,0.3){$0$};
\node[black] at(5.1,0.3){$0$};
\node[black] at(5.3,0.3){$0$};
\node[black] at(5.5,0.3){$0$};
\node[black] at(5.9,0.3){$\cdots$};
\draw  [black](-2,-1)--(7,-1);
\filldraw [black] (1.5,-1) circle (2pt);
\node (b) at(-1.5,-0.7){\textcolor{red}{The $b$-ary expansion of $y$}};
\filldraw [black] (2.5,-1) circle (2pt);
\filldraw [black] (5.5,-1) circle (2pt);
\node[black] at(1.5,-0.7){$a^*$};
\node[black] at(1.7,-0.7){$0$};
\node[black] at(2.1,-0.7){$\cdots$};
\node[black] at(2.5,-0.7){$0$};
\node[black] at(2.7,-0.7){$0$};
\node[black] at(2.9,-0.7){$0$};
\node[black] at(3.1,-0.7){$0$};
\node[black] at(3.3,-0.7){$0$};
\node[black] at(3.5,-0.7){$0$};
\node[black] at(3.7,-0.7){$0$};
\node[black] at(3.9,-0.7){$0$};
\node[black] at(4.3,-0.7){$\cdots$};
\node[black] at(4.7,-0.7){$0$};
\node[black] at(4.9,-0.7){$0$};
\node[black] at(5.1,-0.7){$0$};
\node[black] at(5.3,-0.7){$0$};
\node[black] at(5.5,-0.7){$0$};
\node[black] at(5.9,-0.7){$\cdots$};
\node (b) at(-1.5,-1.7){\textcolor{red}{The $b$-ary expansion of $\frac{k_2}{b^{i}}+y$}};
\draw  [black](-2,-2)--(7,-2);
\filldraw [black] (1.5,-2) circle (2pt);
\node (b) at(1.5,-2.5){$\varepsilon_{m}$};
\filldraw [black] (2.5,-2) circle (2pt);
\filldraw [black] (5.5,-2) circle (2pt);
\node (c) at(2.5,-2.5){$\varepsilon_{i}$};
\node (c) at(5.5,-2.5){$\varepsilon_{j}$};
\node[black] at(2.5,-1.7){$a'$};
\node[black] at(2.7,-1.7){$0$};
\node[black] at(2.9,-1.7){$0$};
\node[black] at(3.1,-1.7){$0$};
\node[black] at(3.3,-1.7){$0$};
\node[black] at(3.5,-1.7){$0$};
\node[black] at(3.7,-1.7){$0$};
\node[black] at(3.9,-1.7){$0$};
\node[black] at(4.3,-1.7){$\cdots$};
\node[black] at(4.7,-1.7){$0$};
\node[black] at(4.9,-1.7){$0$};
\node[black] at(5.1,-1.7){$0$};
\node[black] at(5.3,-1.7){$0$};
\node[black] at(5.5,-1.7){$0$};
\node[black] at(5.9,-1.7){$\cdots$};
\end{tikzpicture}
\end{scaletikzpicturetowidth}
\caption{} \label{bianryi1>m}
\end{figure}

Hence, depending on the scale $i$ at which $L_{\alpha_1}^b$  exhibits a jump, the corresponding scale of discontinuity in $L_{\alpha_2}^{b,y}$ is determined. This completes the proof of Lemma \ref{bi_oscillation_rational}.
\end{proof}

Next, we give the results of oscillations in the $b$-adic interval $\lambda= \left[\frac{k'}{b^j},\frac{k'+1}{b^j}\right)$ of L\'evy functions $L_{\alpha_1}^{b}$ and $L_{\alpha_2}^{b,y}$ when $y$ is not  a $b$-adic rational number. Note that the scale of $\lambda$ is $j$. By the definition of $u_k (\forall k \geqslant 2)$ in \eqref{rst}, there exist $k$ such that $j\in [u_k,u_{k+2})$, and the blocks $\varepsilon_{u_k-1}\varepsilon_{n_k}\cdots\varepsilon_{u_{k+2}-1}$ in the $b$-ary expansion of $y$ satisfying one of the following cases: 
\begin{encase}
\item: \label{irration1}
$
        \forall i= u_{k}-1,\ \text{and}\ u_{k+1}\leqslant i< u_{k+2},\ \varepsilon_{i}\neq 0;\quad \text{and} \quad \forall u_{k}\leqslant i< u_{k+1}, \ \varepsilon_{i}=0.
$

\item: \label{irration3}
$
        \forall i= u_{k}-1,\ \text{and}\ u_{k+1}\leqslant i< u_{k+2},\ \varepsilon_{i}\neq b-1;\quad \text{and} \quad \forall u_{k}\leqslant i< u_{k+1}, \ \varepsilon_{i}=b-1.
$
\end{encase}
\begin{lemma} \label{bi_oscillation_irrational}
Suppose that $y$ is not a $b$-adic rational number and let  $\lambda= \left[\frac{k'}{b^j},\frac{k'+1}{b^j}\right)$ be a $b$-adic interval. If the oscillation $d_{\lambda}^{(2)}$ in $3\lambda$ of the translated L\'evy function $L_{\alpha_2}^{b,y}$ is $d_{\lambda}^{(2)}\sim b^{-\alpha_2 i}$, then the oscillation $d_{\lambda}^{(1)}$ in $3\lambda$ of the L\'evy function $L_{\alpha_1}^b$  satisfies that, for $j\in[u_k,u_{k+1})$,
\begin{equation}
    \nonumber
    d_{\lambda}^{(1)}\sim 
    \left\{
             \begin{array}{ll}
           b^{-\alpha_1 u_k},&\quad \text{\rm if}\ i< u_k-1\\
            b^{- \alpha_1 i'},\ \text{\rm with}\ 1\leqslant i'\leqslant u_k-1,&\quad \text{\rm if}\ i=u_k-1\\
             b^{- \alpha_1 i},&\quad \text{\rm if}\ u_k-1\leqslant i\leqslant j;
             \end{array}
\right.
\end{equation}
for $j\in[u_{k+1},u_{k+2})$,
\begin{equation}
    \nonumber
    d_{\lambda}^{(1)}\sim 
    \left\{
             \begin{array}{ll}
           b^{-\alpha_1 u_k},&\quad \text{\rm if}\ i< u_k-1\\
            b^{- \alpha_1 i'},&\quad \text{\rm if}\ i=u_k-1
            ,\ \text{\rm and for all}\ 1\leqslant i'\leqslant u_k-1\\
             b^{- \alpha_1 i},&\quad \text{\rm if}\ u_k\leqslant i<u_{k+1}
            \\
             b^{- \alpha_1 i'},&\quad \text{\rm if}\ u_{k+1}\leqslant i\leqslant j,\ \text{\rm and for all}\ 1\leqslant i'\leqslant i,.
             \end{array}
\right.
\end{equation}
\end{lemma}
\begin{proof}

Observe that $d_{\lambda}^{(2)}\sim\Delta\left(\frac{k_2}{b^{i}}\right)$. In Case 1, the first $b$-adic rational number at which $L_{\alpha_2}^{b,y}$ is discontinuous within $3\lambda$ is $\frac{k_2}{b^i}$, whose $b$-ary expansion satisfies $\varepsilon_i\neq 0$ and $\varepsilon_t=0$ for all $t>i$. Recall that $j$ is the scale of the $b$-adic interval $\lambda$. It holds that the smallest $b$-adic rational number in $\lambda$ is $\frac{k}{b^j}$.  It follows that $i\leqslant j$. As for Case 2, it is sufficient to replace $0$ by $b-1$ in the following proof to obtain the same result.

We then determine the scale of jump of $L_{\alpha_1}^{b}$ in $3\lambda$ which yield the oscillation $d_{\lambda}^1$ by considering the $b$-ary expansion of $\frac{k_2}{b^{i}}+y$.

If $i<u_k-1$, then the $b$-ary expansion of $\frac{k_2}{b^{i}}+y$ satisfies $\varepsilon_{u_k-1}\neq 0$, and $\varepsilon_t=0,$ for $u_k \leqslant t \leqslant u_{k+1}$. As shown in the following figure. It follows that the discontinuity of $L_{\alpha_1}^{b}$ in $3\lambda$ occurs at scale  ${u_k}$. According to Proposition \ref{Propsecondorderoscillation}, for $j\in[u_k,u_k+2)$, the oscillation $d_{\lambda}^{(1)}$ satisfies 
\begin{equation}
    \nonumber
     d_{\lambda}^{(1)}\sim\Delta\left(\frac{k_1}{b^{u_k}}\right)\sim b^{-\alpha_1 u_k}.
\end{equation}
\begin{figure}[H]
\centering
\begin{scaletikzpicturetowidth}{\textwidth}
\begin{tikzpicture}[scale=\tikzscale]
\draw  [black](-2,0)--(7,0);
\filldraw [black] (1.5,0) circle (2pt);
\node (b) at(-1.5,0.3){\textcolor{red}{The $b$-ary expansion of $\frac{k_2}{b^{i}}(i<u_k)$}};
\filldraw [black] (5.5,0) circle (2pt);
\node[black] at(0.9,0.3){$a'$};
\node[black] at(1.5,0.3){$0$};
\node[black] at(1.7,0.3){$0$};
\node[black] at(1.9,0.3){$0$};
\node[black] at(2.1,0.3){$0$};
\node[black] at(2.3,0.3){$0$};
\node[black] at(2.7,0.3){$\cdots$};
\node[black] at(3.1,0.3){$0$};
\node[black] at(3.3,0.3){$0$};
\node[black] at(3.5,0.3){$0$};
\node[black] at(3.7,0.3){$0$};
\node[black] at(3.9,0.3){$0$};
\node[black] at(4.3,0.3){$\cdots$};
\node[black] at(4.7,0.3){$0$};
\node[black] at(4.9,0.3){$0$};
\node[black] at(5.1,0.3){$0$};
\node[black] at(5.3,0.3){$0$};
\node[black] at(5.5,0.3){$0$};
\node[black] at(5.9,0.3){$\cdots$};
\draw  [black](-2,-1)--(7,-1);
\filldraw [black] (1.5,-1) circle (2pt);
\node (b) at(-1.5,-0.7){\textcolor{red}{The $b$-ary expansion of $y$}};
\filldraw [black] (5.5,-1) circle (2pt);
\node[black] at(1.5,-0.7){$a^*$};
\node[black] at(1.7,-0.7){$0$};
\node[black] at(1.9,-0.7){$0$};
\node[black] at(2.1,-0.7){$0$};
\node[black] at(2.3,-0.7){$0$};
\node[black] at(2.5,-0.7){$0$};
\node[black] at(2.7,-0.7){$0$};
\node[black] at(2.9,-0.7){$0$};
\node[black] at(3.1,-0.7){$0$};
\node[black] at(3.3,-0.7){$0$};
\node[black] at(3.5,-0.7){$0$};
\node[black] at(3.7,-0.7){$0$};
\node[black] at(3.9,-0.7){$0$};
\node[black] at(4.3,-0.7){$\cdots$};
\node[black] at(4.7,-0.7){$0$};
\node[black] at(4.9,-0.7){$0$};
\node[black] at(5.1,-0.7){$0$};
\node[black] at(5.3,-0.7){$0$};
\node[black] at(5.5,-0.7){$0$};
\node[black] at(5.9,-0.7){$\cdots$};
\node (b) at(-1.5,-1.7){\textcolor{red}{The $b$-ary expansion of $\frac{k_2}{b^{i}}+y$}};
\draw  [black](-2,-2)--(7,-2);
\filldraw [black] (1.5,-2) circle (2pt);
\node (b) at(1.5,-2.5){$\varepsilon_{u_k-1}$};
\filldraw [black] (5.5,-2) circle (2pt);
\node (c) at(5.5,-2.5){$\varepsilon_{u_{k+1}}$};
\node (b) at(0.9,-2.5){$\varepsilon_{i}$};
\node[black] at(1.5,-1.7){$a^*$};
\node[black] at(1.7,-1.7){$0$};
\node[black] at(1.9,-1.7){$0$};
\node[black] at(2.1,-1.7){$0$};
\node[black] at(2.3,-1.7){$0$};
\node[black] at(2.5,-1.7){$0$};
\node[black] at(2.7,-1.7){$0$};
\node[black] at(2.9,-1.7){$0$};
\node[black] at(3.1,-1.7){$0$};
\node[black] at(3.3,-1.7){$0$};
\node[black] at(3.5,-1.7){$0$};
\node[black] at(3.7,-1.7){$0$};
\node[black] at(3.9,-1.7){$0$};
\node[black] at(4.3,-1.7){$\cdots$};
\node[black] at(4.7,-1.7){$0$};
\node[black] at(4.9,-1.7){$0$};
\node[black] at(5.1,-1.7){$0$};
\node[black] at(5.3,-1.7){$0$};
\node[black] at(5.5,-1.7){$0$};
\node[black] at(5.9,-1.7){$\cdots$};
\end{tikzpicture}
\end{scaletikzpicturetowidth}
\caption{} \label{th2}
\end{figure} 
If $i=u_k-1$, then there will be two subcases of the $b$-ary expansion of $\frac{k_2}{b^i}+y$. Note that $\varepsilon_{i'}\in \mathcal{A}(\ \forall 1\leqslant i' \leqslant i)$ in the $b$-ary expansion of $\frac{k_2}{b^i}$. Let $\bar{a},\tilde{a}\in \mathcal{A}\backslash\{0\}$.
\begin{itemize}
    \item If $\varepsilon_{u_k-1}=\bar{a}$ in the $b$-ary expansion of $y$, and $\varepsilon_m=b-\bar{a}$ in the $b$-ary expansion of $\frac{k_2}{b^i}$,  then the sum $\frac{k_2}{b^i}+y$ may induce a carry in the $b$-ary expansion, where the interaction between the digits of $y$ and $\frac{k_2}{b^i}$  can propagate up to digit $u_k$. Consequently, the $b$-ary expansion of $\frac{k_2}{b^i}+y$ satisfies $\varepsilon_{i'}\neq 0$, and $\varepsilon_t=0,$ for $i' \leqslant t \leqslant u_{k+1}$ with $1\leqslant i'\leqslant u_k-2$. Then the first discontinuity in $L_{\alpha_1}^b$ may occur at scale $i'$ for all $i'\in [1,u_k-2]$
    \item If $\varepsilon_{u_k-1}=\bar{a}$ in the $b$-ary expansion of $y$, and $\varepsilon_m =\tilde{a}$  in the $b$-ary expansion of $\frac{k_2}{b^i}$, which satisfies $\tilde{a}\neq b-\bar{a}$, then
    the $b$-ary expansion of $\frac{k_2}{b^i}+y$ satisfied $\varepsilon_{u_k-1}\neq 0$, and $\varepsilon_t=0,$ for $u_k \leqslant t \leqslant u_{k+1}$. Thus, the corresponding discontinuity in $L_{\alpha_1}^{b}$
  appears at scale $u_k-1$.
\end{itemize}  
Thus, the oscillation $d_{\lambda}^{(1)}\sim\Delta\left(\frac{k_1}{b^{i'}}\right)\sim b^{-\alpha i'}$ with $1\leqslant i'\leqslant u_k-1$ for $j\in [u_k,u_{k+2})$.

If $u_k\leqslant i< u_{k+1}$, then the $b$-ary expansion of $\frac{k_2}{b^{i}}+y$ satisfies $\varepsilon_{i}\neq 0$, and $\varepsilon_t=0,$ for $i < t \leqslant u_{k+1}$. As shown in Figure \ref{th1}.  Combining Proposition \ref{Propsecondorderoscillation}, the following results of oscillation $d_{\lambda}^{(1)}$ hold. When $u_k \leqslant j< u_{k+1}\ ( \forall\ k\geqslant 1)$, it follows that the first discontinuity of $L_{\alpha_1}^{b}$ in $3\lambda$ is $\frac{k_1}{b^{i}}$ with $i\in[u_k,j]$. Then, the oscillation $d_{\lambda}^{(1)}$ satisfies that for $i\in[u_k-1,j)$
\begin{equation}
    \nonumber
     d_{\lambda}^{(1)}\sim\Delta\left(\frac{k_1}{b^{i}}\right)\sim b^{-\alpha_1 i}.
\end{equation}
As a result, when $u_{k+1} \leqslant j< u_{k+2}\ ( \forall\ k\geqslant 1),$  the discontinuity of $L_{\alpha_1}^{b}$ in $3\lambda$ occurs at scale $i$ with $i\in[u_k,u_{k+1}]$. Then, the oscillation $d_{\lambda}^{(1)}$ satisfies that for $i\in[u_k-1,u_{k+1})$
\begin{equation}
    \nonumber
     d_{\lambda}^{(1)}\sim\Delta\left(\frac{k_1}{b^{i}}\right)\sim b^{-\alpha_1 i}.
\end{equation}
\begin{figure}[H]
\centering
\begin{scaletikzpicturetowidth}{\textwidth}
\begin{tikzpicture}[scale=\tikzscale]
\draw  [black](-2,0)--(7,0);
\filldraw [black] (1.5,0) circle (2pt);
\node (b) at(-1.5,0.3){\textcolor{red}{The $b$-ary expansion of $\frac{k_2}{b^{i}}(u_k\leqslant i<v_k)$}};
\filldraw [black] (5.5,0) circle (2pt);
\node[black] at(2.1,0.3){$a'$};
\node[black] at(2.3,0.3){$0$};
\node[black] at(2.7,0.3){$\cdots$};
\node[black] at(3.1,0.3){$0$};
\node[black] at(3.3,0.3){$0$};
\node[black] at(3.5,0.3){$0$};
\node[black] at(3.7,0.3){$0$};
\node[black] at(3.9,0.3){$0$};
\node[black] at(4.3,0.3){$\cdots$};
\node[black] at(4.7,0.3){$0$};
\node[black] at(4.9,0.3){$0$};
\node[black] at(5.1,0.3){$0$};
\node[black] at(5.3,0.3){$0$};
\node[black] at(5.5,0.3){$0$};
\node[black] at(5.9,0.3){$\cdots$};
\draw  [black](-2,-1)--(7,-1);
\filldraw [black] (1.5,-1) circle (2pt);
\node (b) at(-1.5,-0.7){\textcolor{red}{The $b$-ary expansion of $y$}};
\filldraw [black] (5.5,-1) circle (2pt);
\node[black] at(1.5,-0.7){$a^*$};
\node[black] at(1.7,-0.7){$0$};
\node[black] at(1.9,-0.7){$0$};
\node[black] at(2.1,-0.7){$0$};
\node[black] at(2.3,-0.7){$0$};
\node[black] at(2.5,-0.7){$0$};
\node[black] at(2.7,-0.7){$0$};
\node[black] at(2.9,-0.7){$0$};
\node[black] at(3.1,-0.7){$0$};
\node[black] at(3.3,-0.7){$0$};
\node[black] at(3.5,-0.7){$0$};
\node[black] at(3.7,-0.7){$0$};
\node[black] at(3.9,-0.7){$0$};
\node[black] at(4.3,-0.7){$\cdots$};
\node[black] at(4.7,-0.7){$0$};
\node[black] at(4.9,-0.7){$0$};
\node[black] at(5.1,-0.7){$0$};
\node[black] at(5.3,-0.7){$0$};
\node[black] at(5.5,-0.7){$0$};
\node[black] at(5.9,-0.7){$\cdots$};
\node (b) at(-1.5,-1.7){\textcolor{red}{The $b$-ary expansion of $\frac{k_2}{b^{i}}+y$}};
\draw  [black](-2,-2)--(7,-2);
\filldraw [black] (1.5,-2) circle (2pt);
\node (b) at(1.5,-2.5){$\varepsilon_{u_k-1}$};
\filldraw [black] (5.5,-2) circle (2pt);
\node (c) at(5.5,-2.5){$\varepsilon_{u_{k+1}}$};
\node (b) at(2.1,-2.5){$\varepsilon_{i}$};
\node[black] at(2.1,-1.7){$a'$};
\node[black] at(2.3,-1.7){$0$};
\node[black] at(2.5,-1.7){$0$};
\node[black] at(2.7,-1.7){$0$};
\node[black] at(2.9,-1.7){$0$};
\node[black] at(3.1,-1.7){$0$};
\node[black] at(3.3,-1.7){$0$};
\node[black] at(3.5,-1.7){$0$};
\node[black] at(3.7,-1.7){$0$};
\node[black] at(3.9,-1.7){$0$};
\node[black] at(4.3,-1.7){$\cdots$};
\node[black] at(4.7,-1.7){$0$};
\node[black] at(4.9,-1.7){$0$};
\node[black] at(5.1,-1.7){$0$};
\node[black] at(5.3,-1.7){$0$};
\node[black] at(5.5,-1.7){$0$};
\node[black] at(5.9,-1.7){$\cdots$};
\end{tikzpicture}
\end{scaletikzpicturetowidth}
\caption{} \label{th1}
\end{figure} 

If $u_{k+1}\leqslant i< u_{k+2}$, then taking the same analysis as $i=u_{k}-1 $, it follows that the $b$-ary expansion of $\frac{k_2}{b^{i}}+y$ satisfies $\varepsilon_{i'}\neq 0$, and $\varepsilon_t=0,$ for $i' < t \leqslant u_{k+2}$ with $1\leqslant i'\leqslant i$. Combining that the scale of $\lambda$ is j, it holds that for $i\in[u_{k+1},j]$, the oscillation $d_{\lambda}^{(1)}\sim\Delta\left(\frac{k_1}{b^{i'}}\right)\sim b^{-\alpha i'}$ with $1\leqslant i'\leqslant i$ for $j\in [u_{k+1},u_{k+2})$.

The lemma then follows by analyzing all four cases based on the position of $i$ relative to the digits of the expansion of $y$.
\end{proof}

After obtaining the oscillations $d_{\lambda}^{(1)}$ and $d_{\lambda}^{(2)}$ of L\'evy functions $L_{\alpha_1}^b$ and $L_{\alpha_2}^{b,y}$ in $3\lambda$, the following lemmas give the results of the scaling function $S(r,j)$ as defined in \eqref{scalingfunction}. 
\begin{lemma} \label{lemma_structurefunc_rational}
Let $y$ be a $b$-adic rational number. Then, the scaling function $\zeta(r,j)$ for $L_{\alpha_1}^b$ and $L_{\alpha_2}^{b,y}$ is given by 
\begin{equation} \label{scalingfunctionrational}
    \begin{aligned}
      \zeta(r)=
      \left\{
             \begin{array}{ll}
            \alpha_1r_1+\alpha_2r_2,&\quad \text{if}\ 1-\alpha_1 r_1 -\alpha_2 r_2\geqslant 0,\vspace{1ex}\\
            1,&\quad \text{if}\ 1-\alpha_1 r_1 -\alpha_2 r_2 < 0,
             \end{array}
\right.
    \end{aligned}
\end{equation}
as shown below.
\end{lemma}
\begin{figure}[H]
\centering
\begin{tikzpicture}[scale=0.5]
\fill[blue!40] (-2,6) -- (7,-3)--(8,-3)--(8,6)--(-2,6);
 \fill[green!40] (-2,6) -- (7,-3)--(-3,-3)--(-3,6)--(-2,6);
\draw [-stealth](-3,0) -- (8,0)node[below,xshift=0.8em,yshift=0.5em]{$r_1$};
\draw [-stealth](1,-3) -- (1,6)node[below,xshift=0em,yshift=1em]{$r_2$};
\node [black] at(3.8,-0.4){$\frac{1}{\alpha_1}$};
\node [black] at(0.6,2.8){$\frac{1}{\alpha_2}$};
\draw[thick,dashed] (-2,6) -- (7,-3);
\draw [->] (-1.5,5.5) -- (-1.5,7);
\node [black] at(-1.5,7.5){$1-\alpha_1r_1-\alpha_2r_2=0$};
\node [red] at(4.9,3.5){$\zeta(r)=1$};
\node [red] at(0.6,0.5){$\zeta(r)=\alpha_1r_1+\alpha_2r_2$};
\end{tikzpicture}
\caption{Scaling function ($y$ is $b$-adic rational)}
\end{figure}
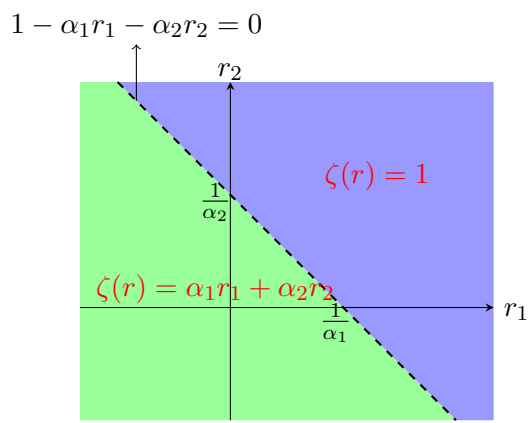
\begin{proof}
Based on Lemma \ref{bi_oscillation_rational}, the oscillations $d_{\lambda}^{(1)}$ and $d_{\lambda}^{(2)}$ of  $L_{\alpha_1}^b$ and $L_{\alpha_2}^{b,y}$ over the interval $3\lambda$ satisfy the following relations: if $i< m$, then
\begin{equation} \nonumber
  \begin{aligned}
      d_{\lambda}^1\sim\Delta\left(\frac{k_1}{b^{m}}\right),\ 
      d_{\lambda}^2\sim\Delta\left(\frac{k_2}{b^{i}}\right);
  \end{aligned}
\end{equation}
if $i=m$, then for all $ 1\leqslant i'\leqslant m$,
\begin{equation} \nonumber
  \begin{aligned}
      d_{\lambda}^1\sim\Delta\left(\frac{k_1}{b^{i'}}\right),\  
      d_{\lambda}^2\sim\Delta\left(\frac{k_2}{b^{m}}\right);
  \end{aligned}
\end{equation}
if $i> m$, then
\begin{equation} \nonumber
  \begin{aligned}
      d_{\lambda}^1\sim\Delta\left(\frac{k_1}{b^{i}}\right),\  
     d_{\lambda}^2\sim\Delta\left(\frac{k_2}{b^{i}}\right).
  \end{aligned}
\end{equation}

We now evaluate the structure function $S(r,j)$, defined in \eqref{sructurefunction}. For every $r=(r_1,r_2)\in \mathbb{R}^2$, it takes the form 
\begin{equation} \nonumber
    S(r,j)=b^{-j}\sum\limits_{\lambda\in\Lambda_j}\big(d_{\lambda}^{(1)}\big)^{r_1}\big(d_{\lambda}^{(2)}\big)^{r_2}.
\end{equation}
 Asymptotically, this expression simplifies to 
\begin{equation} \nonumber
\begin{aligned}
      S(r,j)&\sim b^{-j}\left[\sum\limits_{i=1}^{m}b^ib^{- \alpha_1 r_1m}b^{- \alpha_2 r_2 i}+\sum\limits_{i=1}^{m}b^ib^{- \alpha_1 r_1i}b^{- \alpha_2 r_2 m} +\sum\limits_{i=m}^{j}b^ib^{- \alpha_1 r_1i}b^{- \alpha_2 r_2 i}\right]
      \\&\sim b^{-j}\left[b^{-\alpha_1 r_1 m}\sum\limits_{i=0}^{m}b^{i(1- \alpha_2 r_2)}+ b^{-\alpha_2 r_2 m}\sum\limits_{i=0}^{m}b^{i(1- \alpha_1 r_1)}+\sum\limits_{i=m}^{j}b^{i(1- \alpha_1 r_1- \alpha_2 r_2)}\right].
\end{aligned}
\end{equation}
Since the scale of the interval $j$ is much larger than $m$, it follows  that the "main part" in the structure function $S(r,j)$ is $b^{-j}\sum\limits_{i=m}^{j}b^{i(1- \alpha_1 r_1- \alpha_2 r_2)}$.   Here, the "main part" refers to the dominant contribution as $j\to \infty$, meaning that any remaining terms are of smaller order, i.e., they satisfy $b^{-\alpha_1 r_1 m}\sum\limits_{i=0}^{m}b^{i(1- \alpha_2 r_2)}+ b^{-\alpha_2 r_2 m}\sum\limits_{i=0}^{m}b^{i(1- \alpha_1 r_1)}=o(b^{-\alpha_2 r_2 m}\sum\limits_{i=0}^{m}b^{i(1- \alpha_1 r_1)}+\sum\limits_{i=m}^{j}b^{i(1- \alpha_1 r_1- \alpha_2 r_2)})$. This notation will be used consistently throughout the paper, where the "main part" of an expression will always refer to its leading-order term in the asymptotic expansion. Hence, the structure function simplifies to
\begin{equation}
    \nonumber
     S(r,j)\sim b^{-j}\sum\limits_{i=0}^{j}b^{i(1- \alpha_1 r_1- \alpha_2 r_2)}.
\end{equation}

After that, the scaling function $\zeta(r)$ as defined in \eqref{scalingfunction} is as follows. 
\begin{itemize}
    \item If $1-\alpha_1 r_1-\alpha_2 r_2 \geqslant 0,$ then 
    \begin{equation}
        \nonumber
        \sum\limits_{i=0}^{j}b^{i(1- \alpha_1 r_1- \alpha_2 r_2)}\sim b^{j(1- \alpha_1 r_1- \alpha_2 r_2)},
    \end{equation}
which yields the scaling function
\begin{equation}\nonumber
\zeta(r)=\liminf\limits_{j\to+\infty}\frac{\log(S(r,j))}{\log(b^{-j})}=\liminf\limits_{j\to+\infty}\frac{\log\left(b^{-j+j(1- \alpha_1 r_1-\alpha_2 r_2)}  \right)}{\log(b^{-j})}
\end{equation}
\begin{equation}\nonumber
     =\liminf\limits_{j\to+\infty}\frac{\log\left(b^{-j(\alpha_1 r_1+\alpha_2 r_2)}\right)}{\log(b^{-j})}
     =\alpha_1 r_1+\alpha_2 r_2.
\end{equation}
    \item If $1-\alpha_1 r_1 -\alpha_2 r_2< 0$, then the summation remains bounded, that is,
    \begin{equation}
        \nonumber
        \sum\limits_{i=0}^{j}b^{i(1- \alpha_1 r_1- \alpha_2 r_2)}\sim 1,
    \end{equation}
and it follows that,
\begin{equation}\nonumber
\begin{aligned}
\zeta(r)=\liminf\limits_{j\to+\infty}\frac{\log(S(r,j))}{\log(b^{-j})}
=\liminf\limits_{j\to+\infty}\frac{\log\left(b^{-j}  \right)}{\log(b^{-j})}=1.
\end{aligned}
\end{equation}
\end{itemize}
Therefore, the scaling function is given by
\begin{equation} \nonumber
    \begin{aligned}
      \zeta(r)=
      \left\{
             \begin{array}{ll}
            \alpha_1r_1+\alpha_2r_2,&\quad \text{if}\ 1-\alpha_1 r_1 -\alpha_2 r_2\geqslant 0,\vspace{1ex}\\
            1,&\quad \text{if}\ 1-\alpha_1 r_1 -\alpha_2 r_2 < 0.
             \end{array}
\right.
    \end{aligned}
\end{equation}
\end{proof}

\begin{lemma} \label{lemma_structurefunc_irrational}
Suppose that $y$ is not a $b$-adic rational number. Let $\{w_k\}$  denote the sequence defined in  \eqref{w_k}.  Then the associated scaling function for $L_{\alpha_1}^b$ and $L_{\alpha_2}^{b,y}$ is as follows: If $\{w_k\}$ is bounded, then 
\begin{equation} \label{scalingfunctionirrational}
    \begin{aligned}
      \zeta(r)=
      \left\{
             \begin{array}{ll}
            \alpha_1r_1+\alpha_2r_2,&\quad \text{if}\ 1-\alpha_1 r_1 \geqslant 0\ \text{and}\  1-\alpha_2 r_2 \geqslant 0,\vspace{1ex}\\
            1+\alpha_1r_1,&\quad \text{if}\ 1-\alpha_2 r_2 < 0\ \text{and}\  \alpha_1 r_1\leqslant  \alpha_2 r_2,\vspace{1ex}
            \\
            1+\alpha_2r_2,&\quad \text{if}\ 1-\alpha_1 r_1 < 0\ \text{and}\  \alpha_1 r_1>  \alpha_2 r_2;
             \end{array}
\right.
    \end{aligned}
\end{equation}
as shown in Figure \ref{graph_scalingfunction_bound}. If $\{w_k\}$ is unbounded, then
\begin{equation} \nonumber
    \begin{aligned}
      \zeta(r)=
      \left\{
             \begin{array}{ll}
            \alpha_1r_1+\alpha_2r_2,&\quad \text{if}\ 1-\alpha_1 r_1 \geqslant 0,\ 1-\alpha_2 r_2 \geqslant 0\ \text{and}\  1-\alpha_1 r_1-\alpha_2 r_2 \geqslant 0,\vspace{1ex}\\
            1+\alpha_1r_1,&\quad \text{if}\ 1-\alpha_2 r_2 < 0\ \text{and}\  r_1< 0,\vspace{1ex}
            \\
            1+\alpha_2r_2,&\quad \text{if}\ 1-\alpha_1 r_1 < 0\ \text{and}\ r_2 < 0,\vspace{1ex}
            \\
            1,&\quad \text{if}\ 1-\alpha_1 r_1 -\alpha_2 r_2 < 0, \ r_1\geqslant 0\ \text{and}\ r_2 \geqslant 0,
             \end{array}
\right.
    \end{aligned}
\end{equation}
as shown in Figure \ref{graph_scalingfunction_unbound}.
\end{lemma}

\begin{figure}[ht]
\centering
\begin{subfigure}{0.48\textwidth}
\centering
\begin{tikzpicture}[scale=0.45]
\fill[blue!40] (-3,6) -- (8,6)--(4,3)--(-3,3)--(-3,6);
\fill[green!40] (4,-3) -- (8,-3)--(8,6)--(4,3)--(4,-3);
\fill[red!40] (-3,3) -- (1,3)--(4,3)--(4,-3)--(-3,-3);
\draw [-stealth](-3,0) -- (8,0)node[below,xshift=0.8em,yshift=0.5em]{$r_1$};
\draw [-stealth](1,-3) -- (1,6)node[below,xshift=0em,yshift=1em]{$r_2$};
\node [black] at(3.8,-0.4){$\frac{1}{\alpha_1}$};
\node [black] at(0.6,2.8){$\frac{1}{\alpha_2}$};
\draw[thick,dashed] (4,3) -- (8,6);
\draw [->] (7.7,5.8) -- (8.6,5.8);
\node [black] at(10.8,5.8){$\alpha_1r_1=\alpha_2r_2$};
\node [red] at(0.9,4.5){$\zeta(r)=1+\alpha_1r_1$};
\node [red] at(7,0.5){$\zeta(r)=1+\alpha_2r_2$};
\node [red] at(0.6,0.5){$\zeta(r)=\alpha_1r_1+\alpha_2r_2$};
\end{tikzpicture}
\subcaption{$\{w_k\}$ bounded}
\label{graph_scalingfunction_bound}
\end{subfigure}
\hfill
\begin{subfigure}{0.48\textwidth}
\centering
\begin{tikzpicture}[scale=0.45]
\fill[red!40] (-3,-3)--(4,-3)--(4,0)--(1,3)-- (-3,3);
\fill[green!40] (4,-3)--(8,-3)--(8,0)--(4,0);
\fill[blue!40] (1,3)-- (1,6)--(-3,6)--(-3,3);
\fill[yellow!40] (1,6)-- (1,3)--(4,0)-- (8,0)--(8,6);
\draw[thick,dashed] (-3,3) -- (1,3);
\draw[thick,dashed] (4,-3) -- (4,0);
\draw[thick,dashed] (1,3) -- (4,0);
\node [red] at(0.9,0.2){$\alpha_1r_1+\alpha_2r_2$};  
\node [red] at(6,-1.2){$1+\alpha_2r_2$}; 
\node [red] at(-0.9,4.5){$1+\alpha_1r_1$};  
\node [red] at(5,3){$1$};
\draw [-stealth](-3,0) -- (8,0)node[below,xshift=0.8em,yshift=0.5em]{$r_1$};
\draw [-stealth](1,-3) -- (1,6)node[below,xshift=0em,yshift=1em]{$r_2$};
\node [black] at(3.7,-0.4){$\frac{1}{\alpha_1}$};
\node [black] at(0.6,2.8){$\frac{1}{\alpha_2}$};
\end{tikzpicture}
\subcaption{$\{w_k\}$ unbounded}
\label{graph_scalingfunction_unbound}
\end{subfigure}
\caption{Scaling function ($y$ is non-$b$-adic rational)}
\label{fig:scalingfunction-total}
\end{figure}
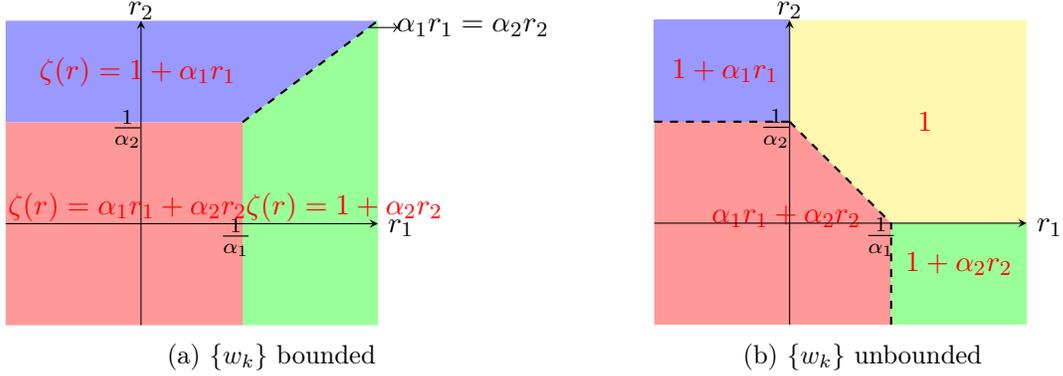

\begin{proof}
Recall from \eqref{rst} that $w_k=u_{k+1}-u_k$. We begin by analyzing the case where the sequence $\{w_k\}$ is bounded. According to Lemma \ref{bi_oscillation_irrational}, the oscillations $d_{\lambda}^{(1)}$ and $d_{\lambda}^{(2)}$ of  $L_{\alpha_1}^b$ and $L_{\alpha_2}^{b,y}$ over $3\lambda$  fall into one of two cases, depending on the dyadic scale $j$ of the interval.

\begin{itemize}
    \item  In the case $j\in[u_k,u_{k+1})$, the oscillation behavior can be classified as follows: if $i<u_k-1$, then
\begin{equation} \nonumber
  \begin{aligned}
      d_{\lambda}^{(1)}\sim\Delta\left(\frac{k_1}{b^{u_k}}\right),  \quad \mbox{ and } \quad  d_{\lambda}^{(2)}\sim\Delta\left(\frac{k_2}{b^{i}}\right);
  \end{aligned}
\end{equation}
if $i=u_k-1$, then
\begin{equation} \nonumber
  \begin{aligned}
      \forall 1\leqslant i'\leqslant u_k-1, \ d_{\lambda}^{(1)}\sim\Delta\left(\frac{k_1}{b^{i'}}\right), \quad \mbox{ and } \quad  d_{\lambda}^{(2)}\sim\Delta\left(\frac{k_2}{b^{u_k-1}}\right);
  \end{aligned}
\end{equation}
if $u_k-1<i\leqslant j$, then
\begin{equation} \nonumber
  \begin{aligned}
      d_{\lambda}^{(1)}\sim\Delta\left(\frac{k_1}{b^{i}}\right),  \quad \mbox{ and } \quad  d_{\lambda}^{(2)}\sim\Delta\left(\frac{k_2}{b^{i}}\right).
  \end{aligned}
\end{equation}
\item When $u_{k+1} < j< u_{k+2}\ (\forall\ k\geqslant 1)$, the oscillations $d_{\lambda}^{(1)}$ and $d_{\lambda}^{(2)}$ are shown below.
If $i<u_k-1$, then
\begin{equation} \nonumber
  \begin{aligned}
      d_{\lambda}^{(1)}\sim\Delta\left(\frac{k_1}{b^{u_k}}\right),  \quad \mbox{ and } \quad  d_{\lambda}^{(2)}\sim\Delta\left(\frac{k_2}{b^{i}}\right);
  \end{aligned}
\end{equation}
if $i=u_k-1$, then
\begin{equation} \nonumber
  \begin{aligned}
     \forall 1\leqslant i'\leqslant u_k-1, \ d_{\lambda}^{(1)}\sim\Delta\left(\frac{k_1}{b^{i'}}\right), \quad \mbox{ and } \quad  d_{\lambda}^{(2)}\sim\Delta\left(\frac{k_2}{b^{u_k-1}}\right);
  \end{aligned}
\end{equation}
if $u_k\leqslant i<u_{k+1}$, then
\begin{equation} \nonumber
  \begin{aligned}
      d_{\lambda}^{(1)}\sim\Delta\left(\frac{k_1}{b^{i}}\right),  \quad \mbox{ and } \quad  d_{\lambda}^{(2)}\sim\Delta\left(\frac{k_2}{b^{i}}\right);
  \end{aligned}
\end{equation}
if $u_{k+1}\leqslant i<j$, then 
\begin{equation} \nonumber
  \begin{aligned}
     \forall 1\leqslant i'\leqslant i,\  d_{\lambda}^{(1)}\sim\Delta\left(\frac{k_1}{b^{i'}}\right),  \quad \mbox{ and } \quad  d_{\lambda}^{(2)}\sim\Delta\left(\frac{k_2}{b^{i}}\right);
  \end{aligned}
\end{equation}
\end{itemize} 

Next, we determine the scaling function under the assumption that the sequence $\{w_k\}$ is bounded. Let $j\in [u_k, u_{k+1})$. According to the definition of the structure function $S(r,j)$ given in in \eqref{sructurefunction}, it follows that
\begin{equation}
    \begin{aligned}
        \label{structurefunctionbound1}
           S(r,j)&=b^{-j}\sum\limits_{\lambda\in\Lambda_j}\big(d_{\lambda}^{(1)}\big)^{r_1}\big(d_{\lambda}^{(2)}\big)^{r_2}
      \\&\sim b^{-j}\left[\sum\limits_{i=1}^{u_k-2}b^ib^{- \alpha_1 r_1 u_k}b^{- \alpha_2 r_2 i}+\sum\limits_{i=1}^{u_k-1}b^ib^{- \alpha_1 r_1i}b^{- \alpha_2 r_2 u_k}+\sum\limits_{i=u_k}^{j}b^ib^{- \alpha_1 r_1i}b^{- \alpha_2 r_2 i} \right]
      \\&\sim b^{-j}\left[b^{-\alpha_1 r_1 u_k}\sum\limits_{i=0}^{u_k}b^{i(1- \alpha_2 r_2)}+ b^{-\alpha_2 r_2 u_k}\sum\limits_{i=0}^{u_k}b^{i(1- \alpha_1 r_1)}+\sum\limits_{i=u_k}^{j}b^{i(1- \alpha_1 r_1- \alpha_2 r_2)}\right].
    \end{aligned}
\end{equation}
Since $\{w_k\}$ is bounded, the "main part" in the structure function is $$b^{-j}\left[b^{-\alpha_1 r_1 u_k}\sum\limits_{i=0}^{u_k}b^{i(1- \alpha_2 r_2)}+ b^{-\alpha_2 r_2 u_k}\sum\limits_{i=0}^{u_k}b^{i(1- \alpha_1 r_1)}\right].$$ Thereafter,  for all $ r = (r_1,r_2)\in \mathbb{R}^2,$  it holds that
\begin{equation} \label{structurefunctionbounded}
\begin{aligned}
      S(r,j)\sim  b^{-j(1+\alpha_1 r_1)}\sum\limits_{i=0}^{j}b^{i(1- \alpha_2 r_2)}+b^{-j(1+\alpha_2 r_2)}\sum\limits_{i=0}^{j}b^{i(1- \alpha_1 r_1)}.
\end{aligned}
\end{equation}
Let $u_{k+1} < j< u_{k+2}$. By the definition of the structure function, it holds that
\begin{equation}
    \begin{aligned}
        \label{structurefunctionbound2}
           S(r,j)&=b^{-j}\sum\limits_{\lambda\in\Lambda_j}\big(d_{\lambda}^{(1)}\big)^{r_1}\big(d_{\lambda}^{(2)}\big)^{r_2}
     \\&\sim b^{-j}\left[\sum\limits_{i=1}^{u_k-2}b^ib^{- \alpha_1 r_1u_k}b^{- \alpha_2 r_2 i}+\sum\limits_{i=1}^{u_k-1}b^ib^{- \alpha_1 r_1i}b^{- \alpha_2 r_2 u_k}+\sum\limits_{i=u_k}^{u_{k+1}}b^ib^{- \alpha_1 r_1i}b^{- \alpha_2 r_2 i} \right.
\\&\quad+\left.\sum\limits_{i'=1}^{i}\sum\limits_{i=u_{k+1}}^{j}b^ib^{- \alpha_1 r_1i'}b^{- \alpha_2 r_2 i}\right]
      \\&= b^{-j}\left[b^{-\alpha_1 r_1 u_k}\sum\limits_{i=0}^{u_k}b^{i(1- \alpha_2 r_2)}+ b^{-\alpha_2 r_2 u_k}\sum\limits_{i=0}^{u_k}b^{i(1- \alpha_1 r_1)}+\sum\limits_{i=u_k}^{u_{k+1}}b^ib^{- \alpha_1 r_1i}b^{- \alpha_2 r_2 i} \right.
\\&\quad+\left.\sum\limits_{i'=1}^{i}\sum\limits_{i=u_{k+1}}^{j}b^ib^{- \alpha_1 r_1i'}b^{- \alpha_2 r_2 i}\right]
    \end{aligned}
\end{equation}
As $\{w_k\}$ is bounded, the "main part" in the structure function is $$b^{-j}\left[b^{-\alpha_1 r_1 u_k}\sum\limits_{i=0}^{u_k}b^{i(1- \alpha_2 r_2)}+ b^{-\alpha_2 r_2 u_k}\sum\limits_{i=0}^{u_k}b^{i(1- \alpha_1 r_1)}+\sum\limits_{i'=1}^{i}\sum\limits_{i=u_{k+1}}^{j}b^ib^{- \alpha_1 r_1i'}b^{- \alpha_2 r_2 i}\right].$$ Hence, for all $ r = (r_1,r_2)\in \mathbb{R}^2,$ the same structure function holds as \eqref{structurefunctionbounded}.

Then, we determine the scaling function $\zeta(r)$ as defined in \eqref{scalingfunction} when $\{w_k\}$ is bounded.
\begin{itemize}
    \item When $1-\alpha_1 r_1 \geqslant 0$ and $ 1-\alpha_2 r_2 \geqslant 0,$ the structure function behaves as \eqref{structurefunctionbounded} is 
    \begin{equation}
        \nonumber
         S(r,j)\sim b^{-j(1+\alpha_1 r_1)}b^{j(1- \alpha_2 r_2)}+ b^{-j(1+\alpha_2 r_2)}b^{j(1- \alpha_1 r_1)}.
    \end{equation}
As a consequence, the scaling function is given by
\begin{equation}\nonumber
\begin{aligned}
     \zeta(r)&=\liminf\limits_{j\to+\infty}\frac{\log(S(r,j))}{\log(b^{-j})}\\&=\liminf\limits_{j\to+\infty}\frac{\log\left(b^{-j(1+\alpha_1 r_1)+j(1- \alpha_2 r_2)} +b^{-j(1+\alpha_2 r_2)+j(1- \alpha_1 r_1)}  \right)}{\log(b^{-j})}
     \\&=\liminf\limits_{j\to+\infty}\frac{\log\left(b^{-j(\alpha_1 r_1+\alpha_2 r_2)}\right)}{\log(b^{-j})}
      \\&=\alpha_1 r_1+\alpha_2 r_2.
\end{aligned}
\end{equation}
    \item When $1-\alpha_1 r_1 < 0$ and $ 1-\alpha_2 r_2 \geqslant 0.$ It holds that
    \begin{equation}
        \nonumber
         S(r,j)\sim b^{-j(1+\alpha_1 r_1)}b^{j(1- \alpha_2 r_2)}+b^{-j(1+\alpha_2 r_2)},
    \end{equation}
    and
    \begin{equation}
        \nonumber
        \alpha_1 r_1+\alpha_2 r_2>1+\alpha_2 r_2.
    \end{equation}
Since the first term decays faster, the second dominates, thus,
\begin{equation}\nonumber
\begin{aligned}
     \zeta(r)&=\liminf\limits_{j\to+\infty}\frac{\log(S(r,j))}{\log(b^{-j})}
     \\&=\liminf\limits_{j\to+\infty}\frac{\log\left(b^{-j(1+\alpha_1 r_1)+j(1- \alpha_2 r_2)}+b^{-j(1+\alpha_2 r_2)}  \right)}{\log(b^{-j})}
     \\&=\liminf\limits_{j\to+\infty}\frac{\log\left(b^{-j(\alpha_1 r_1+\alpha_2 r_2)}+b^{-j(1+\alpha_2 r_2)}\right)}{\log(b^{-j})}
      \\&=1+\alpha_2 r_2.
\end{aligned}
\end{equation}
    \item When $1-\alpha_1 r_1 \geqslant 0$ and $ 1-\alpha_2 r_2<0.$ It follows that
    \begin{equation}
        \nonumber
         S(r,j)\sim b^{-j(1+\alpha_1 r_1)}+b^{-j(1+\alpha_2 r_2)}b^{j(1- \alpha_1 r_1)},
    \end{equation}
    and
    \begin{equation}
        \nonumber
        \alpha_1 r_1+\alpha_2 r_2>1+\alpha_1 r_1.
    \end{equation}
   Here, the second term decays faster than the first, hence,
\begin{equation}\nonumber
\begin{aligned}
     \zeta(r)&=\liminf\limits_{j\to+\infty}\frac{\log(S(r,j))}{\log(b^{-j})}\\&=\liminf\limits_{j\to+\infty}\frac{\log\left(b^{-j(1+\alpha_1 r_1)}+b^{-j(1+\alpha_2 r_2)}b^{j(1- \alpha_1 r_1)}  \right)}{\log(b^{-j})}
     \\&=\liminf\limits_{j\to+\infty}\frac{\log\left(b^{-j(1+\alpha_1 r_1)}+b^{-j(\alpha_1 r_1+\alpha_2 r_2)}\right)}{\log(b^{-j})}
      \\&=1+\alpha_1 r_1.
\end{aligned}
\end{equation}
    \item When $1-\alpha_1 r_1 < 0$ and $ 1-\alpha_2 r_2< 0.$ The structure function satisfies
    \begin{equation}
        \nonumber
         S(r,j)\sim b^{-j(1+\alpha_1 r_1)}+ b^{-j(1+\alpha_2 r_2)},
    \end{equation}
    and
\begin{equation}\nonumber
\begin{aligned}
     \zeta(r)=\liminf\limits_{j\to+\infty}\frac{\log(S(r,j))}{\log(b^{-j})}=\liminf\limits_{j\to+\infty}\frac{\log\left(b^{-j(1+\alpha_1 r_1)} +b^{-j(1+\alpha_2 r_2)}  \right)}{\log(b^{-j})}.
\end{aligned}
\end{equation}
In this case, the slower decaying term dominates. Thus,
\begin{equation} \nonumber
    \begin{aligned}
      \zeta(r)=
      \left\{
             \begin{array}{ll}
            1+\alpha_1r_1,&\quad \text{if}\ \alpha_1 r_1\leqslant  \alpha_2 r_2 ,\vspace{1ex}\\
            1+\alpha_2r_2,&\quad \text{if}\ \alpha_1 r_1 > \alpha_2 r_2.
             \end{array}
\right.
    \end{aligned}
\end{equation}
\end{itemize}
Thus, when $\{w_k\}$ is bounded, the scaling function is given by 
\begin{equation} \nonumber
    \begin{aligned}
      \zeta(r)=
      \left\{
             \begin{array}{ll}
            \alpha_1r_1+\alpha_2r_2,&\quad \text{if}\ 1-\alpha_1 r_1 \geqslant 0\ \text{and}\  1-\alpha_2 r_2 \geqslant 0,\vspace{1ex}\\
            1+\alpha_1r_1,&\quad \text{if}\ 1-\alpha_2 r_2 < 0\ \text{and}\  \alpha_1 r_1\leqslant  \alpha_2 r_2,\vspace{1ex}
            \\
            1+\alpha_2r_2,&\quad \text{if}\ 1-\alpha_1 r_1 < 0\ \text{and}\  \alpha_1 r_1>  \alpha_2 r_2.
             \end{array}
\right.
    \end{aligned}
\end{equation}

\bigskip
We now turn to the case where the sequence $\{w_k\}$ is unbounded and determine the corresponding scaling function $\zeta(r)$. Observe that one can extract an increasing subsequence $\{w_{k_i}\} \subset\{w_{k}\}$ for which 
\begin{equation}\nonumber
    {w_{k_i}}\to\infty \quad \mbox{ when }i\to\infty .
\end{equation}
Without causing any confusion,  the same symbols $\{w_k\}$, $\{u_k\}$ will still be used to substitute the subsequences $\{w_{k_i}\}$, $\{u_{k_i}\}$ for notational simplicity.
Assume that $\varepsilon_{u_k}\varepsilon_{u_k+1}\cdots\varepsilon_{u_{k+1}-1}=0^{u_{k+1}-u_k-2}$, and
\begin{equation}\nonumber
    \frac{u_{k}}{u_{k+1}}\to 0 \quad \mbox{ when } k\to\infty.
\end{equation}

The notation $A\gtrsim_c B$ means that there exists a constant $c>2$ (constants may not be uniform) such that $A\geqslant cB$. The implicit constants $c$ are often suppressed in the following.

If $u_k \lesssim j< u_{k+1}$, then the "main part" in the structure function \eqref{structurefunctionbound1} is $b^{-j}\sum\limits_{i=u_k}^{j}b^{i(1- \alpha_1 r_1- \alpha_2 r_2)}$. Thereafter, for all $ r = (r_1,r_2)\in \mathbb{R}^2,$ the structure satisfies
\begin{equation} \nonumber
\begin{aligned}
      S(r,j)&=b^{-j}\sum\limits_{\lambda\in\Lambda_j}\big(d_{\lambda}^{(1)}\big)^{r_1}\big(d_{\lambda}^{(2)}\big)^{r_2}
      \\&\sim b^{-j}\sum\limits_{i=0}^{j}b^ib^{- \alpha_1 r_1i}b^{- \alpha_2 r_2 i}
      \\& \sim b^{-j}\sum\limits_{i=0}^{j}b^{i(1- \alpha_1 r_1- \alpha_2 r_2)}.
\end{aligned}
\end{equation}
The same scaling function is obtained as \eqref{scalingfunctionrational}.

For other values of  $j$, the same analysis is performed as for $\{w_k\}$ being a bounded sequence, the scaling function is  \eqref{scalingfunctionirrational}.

According to the definition \eqref{scalingfunction}, the scaling function should be the minimum between the results in \eqref{scalingfunctionrational} and \eqref{scalingfunctionirrational}. Hence,
\begin{equation} \nonumber
    \begin{aligned}
      \zeta(r)=
      \left\{
             \begin{array}{ll}
            \alpha_1r_1+\alpha_2r_2,&\quad \text{if}\ 1-\alpha_1 r_1 \geqslant 0,\ 1-\alpha_2 r_2 \geqslant 0\ \text{and}\  1-\alpha_1 r_1-\alpha_2 r_2 \geqslant 0,\vspace{1ex}\\
            1+\alpha_1r_1,&\quad \text{if}\ 1-\alpha_2 r_2 < 0\ \text{and}\  r_1< 0,\vspace{1ex}
            \\
            1+\alpha_2r_2,&\quad \text{if}\ 1-\alpha_1 r_1 < 0\ \text{and}\ r_2 < 0,\vspace{1ex}
            \\
            1,&\quad \text{if}\ 1-\alpha_1 r_1 -\alpha_2 r_2 < 0, \ r_1\geqslant 0\ \text{and}\ r_2 \geqslant 0.
             \end{array}
\right.
    \end{aligned}
\end{equation}
\end{proof}

We are now prepared to proceed with the proof of Theorem \ref{multitheofinal}, which characterizes the bivariate multifractal Legendre spectrum associated with $L_{\alpha_1}^b$  and $L_{\alpha_2}^{b,y}$, denoted by $\mathcal{L}_{L_{\alpha_1}^b,L_{\alpha_2}^{b,y}}(H)=\min\limits_{r} 1-\zeta(r)+H\cdot r$, where $H=(H_1,H_2)\in \mathbb{R}^2$. For notational convenience, denote 
\begin{equation} \label{l(r,h)}
    l(r,H):=1-\zeta(r)+H\cdot r.
\end{equation}
\begin{proof}[Proof of Theorem \ref{multitheofinal}]
We begin by evaluating the bivariate multifractal Legendre spectrum  of $L_{\alpha_1}^b$  and $L_{\alpha_2}^{b,y}$ when $y$ is a $b$-adic rational number. According to Lemma \ref{lemma_structurefunc_rational}, the scaling function piecewise obtained. Substituting this into $l(r,h)$, it follows that
\begin{equation} \nonumber
    \begin{aligned}
      l(r,H)=
      \left\{
             \begin{array}{ll}
            1+r_1(H_1-\alpha_1)+r_2(H_2-\alpha_2),&\quad \text{if}\ 1-\alpha_1 r_1 -\alpha_2 r_2\geqslant 0,\vspace{1ex}\\
            r_1H_1+r_2H_2,&\quad \text{if}\ 1-\alpha_1 r_1-\alpha_2 r_2 < 0.
             \end{array}
\right.
    \end{aligned}
\end{equation} 
By \eqref{mml}, the Legendre spectrum is obtained by minimizing $l(r,H)$ with respect to $r$.
\begin{itemize}
    \item If $H_1 > \alpha_1$,\  or $H_2>\alpha_2$, then the linear term $H\cdot r$ dominates, and taking $r_1\to-\infty,$ and $r_2\to-\infty$ yields $l(r,H)\to-\infty$. Consequently,
\begin{equation} \label{bivariatelegendrespectrem141y}
\mathcal{L}_{L_{\alpha_1}^b,L_{\alpha_2}^{b,y}}(H_1, H_2)= \min\limits_{r}l(r,h)=-\infty.
\end{equation}
\item If $H_1 <0$, or $H_2<0$, then increasing $r_1\to\infty$  or $r_2\to\infty$ forces the inner product $H\cdot r\to -\infty$, while the rest of the expression remains bounded. Hence,
\begin{equation} \label{bivariatelegendrespectrem142y}
\mathcal{L}_{L_{\alpha_1}^b,L_{\alpha_2}^{b,y}}(H_1, H_2)= \min\limits_{r}l(r,h)=-\infty.
\end{equation}
\item If $(H_1, H_2) \in [0,\alpha_1]\times [0,\alpha_2]$, then the function $l(r,h)$ is convex, and the constraint $1-\alpha_1r_1-\alpha_2r_2=0$ characterizes the boundary where the minimum is attained. Solving this constraint yields $r_1=\frac{1-\alpha_2r_2}{\alpha_1}.$ Substituting into $l(r,h)$, it follows that
\begin{equation}
\begin{aligned}
       \label{l(r,h)_rational}
l(r,h)=H_1r_1+H_2r_2=H_1\left(\frac{1-\alpha_2r_2}{\alpha_1}\right)+H_2r_2=\frac{H_1}{\alpha_1}+r_2\left(H_2-\frac{\alpha_2}{\alpha_1}H_1\right).
\end{aligned}
\end{equation}
When $H_2> \frac{\alpha_2}{\alpha_1}H_1$, the coefficient of $r_2$ in \eqref{l(r,h)_rational} is positive. Hence, by letting $r_2\to-\infty$, it holds that $l(r,h)\to -\infty$, which implies
\begin{equation} \label{bivariatelegendrespectrem1432y}
\mathcal{L}_{L_{\alpha_1}^b,L_{\alpha_2}^{b,y}}(H_1, H_2)= \min\limits_{r}l(r,h)=-\infty.
\end{equation}
When $H_2<\frac{\alpha_2}{\alpha_1}H_1$, the coefficient of $r_2$ in \eqref{l(r,h)_rational} is negative. Then, letting $r_2 \to \infty$ forces $l(r,h)\to -\infty$, yielding
\begin{equation} \label{bivariatelegendrespectrem143y}
\mathcal{L}_{L_{\alpha_1}^b,L_{\alpha_2}^{b,y}}(H_1, H_2)= \min\limits_{r}l(r,h)=-\infty.
\end{equation}
When $H_2= \frac{\alpha_2}{\alpha_1}H_1$, the coefficient of $r_2$ in \eqref{l(r,h)_rational} vanishes. As a result, $l(r,h)=\frac{H_1}{\alpha_1}$ becomes constant with respect to $r_2$, and its minimum is attained at any point on the constraint line. Hence,
\begin{equation} \label{bivariatelegendrespectrem144y}
\mathcal{L}_{L_{\alpha_1}^b,L_{\alpha_2}^{b,y}}(H_1, H_2)= \min\limits_{r}l(r,h)=\frac{H_1}{\alpha_1}.
\end{equation}
\end{itemize}
Combining the above cases \eqref{bivariatelegendrespectrem141y},\ \eqref{bivariatelegendrespectrem142y}, \eqref{bivariatelegendrespectrem1432y}, \eqref{bivariatelegendrespectrem143y}, and \eqref{bivariatelegendrespectrem144y}, the Legendre spectrum of the L\'evy function $L_{\alpha_1}^b$ and the translated L\'evy function $L_{\alpha_2}^{b,y}$ when $y$ is a $b$-ary rational number is given by
\begin{equation}
        \nonumber
\mathcal{L}_{L_{\alpha_1}^b,L_{\alpha_2}^{b,y}}(H_1, H_2)= \left\{
             \begin{array}{ll}
           \frac{H_1}{\alpha_1}=\frac{H_2}{\alpha_2},&\quad {\rm for} \quad \big(H_1, \frac{\alpha_2}{\alpha_1} H_1\big) \quad {\rm with} \quad \ H_1\in [0,  \alpha_1 ],\vspace{1ex}\\
            -\infty,&\quad \text{otherwise}.
             \end{array}
\right.
    \end{equation}

\bigskip

Next, we determine the determine the bivariate multifractal Legendre spectrum  of $L_{\alpha_1}^b$  and $L_{\alpha_2}^{b,y}$ when $y$ is not a non-$b$-adic rational number. The scaling function is given by Lemma \ref{lemma_structurefunc_rational}. Substituting into $l(r,h)$, for $\{w_k\}$ as a bounded sequence, it follows that
\begin{equation} \nonumber
    \begin{aligned}
      l(r,H)=
      \left\{
             \begin{array}{ll}
            1+r_1(H_1-\alpha_1)+r_2(H_2-\alpha_2),&\quad \text{if}\ 1-\alpha_1 r_1 \geqslant 0\ \text{and}\  1-\alpha_2 r_2 \geqslant 0,\vspace{1ex}\\
            r_1(H_1-\alpha_1)+r_2H_2,&\quad \text{if}\ 1-\alpha_2 r_2 < 0\ \text{and}\  \alpha_1 r_1\leqslant  \alpha_2 r_2
            ,\vspace{1ex}\\
            r_1H_1+r_2(H_2-\alpha_2),&\quad \text{if}\ 1-\alpha_1 r_1 < 0\ \text{and}\  \alpha_1 r_1>  \alpha_2 r_2.
             \end{array}
\right.
    \end{aligned}
\end{equation} 
The Legendre spectrum $\mathcal{L}_{L_{\alpha_1}^b,L_{\alpha_2}^{b,y}}$ which is $\min\limits_{r}l(r,H)$ is derived as follows.
\begin{itemize}
\item If $H_1 > \alpha_1$,\  or $H_2>\alpha_2$, then the linear term $H\cdot r$ dominates as $r_1\to-\infty,$ and $r_2\to-\infty$, leading to $l(r,H)\to-\infty$. Thus
\begin{equation} \label{fbivariatelegendrespectrem2}
\mathcal{L}_{L_{\alpha_1}^b,L_{\alpha_2}^{b,y}}(H_1, H_2)= \min\limits_{r}l(r,h)=-\infty.
\end{equation}

\item If $H_1 <0$, or $H_2<0$, then $H\cdot r\to -\infty$ as $r_1, r_2 \to +\infty$, while other terms remain bounded. Hence, 
\begin{equation} \label{fbivariatelegendrespectrem3}
\mathcal{L}_{L_{\alpha_1}^b,L_{\alpha_2}^{b,y}}(H_1, H_2)= \min\limits_{r}l(r,h)=-\infty.
\end{equation}

\item If $(H_1,H_2) \in [0,\alpha_1]\times [0,\alpha_2]$, and $\frac{H_1}{\alpha_1}+\frac{H_2}{\alpha_2}-1<0$, then consider the case where $H_1 < -\frac{\alpha_1}{\alpha_2}(H_2-\alpha_2)$. That is,  $1-\alpha_1 r_1 < 0$, $1-\alpha_2 r_2 < 0$, and $\alpha_1 r_1\geqslant  \alpha_2 r_2$. On the constraint $\alpha_1 r_1= \alpha_2 r_2$, the function $l(r,h)$ becomes 
\begin{equation}
    \nonumber
    l(r,h)=r_2\left(\frac{\alpha_2}{\alpha_1}H_1+H_2-\alpha_2\right).
\end{equation}
If the coefficient is negative, i.e., $\frac{H_1}{\alpha_1}+\frac{H_2}{\alpha_2}-1<0$, then letting $r_2\to +\infty$ gives 
\begin{equation} \label{fbivariatelegendrespectrem4}
\mathcal{L}_{L_{\alpha_1}^b,L_{\alpha_2}^{b,y}}(H_1, H_2)= \min\limits_{r}l(r,h)=-\infty.
\end{equation}


\item If $(H_1, H_2) \in [0,\alpha_1]\times [0,\alpha_2]$, and $\frac{H_1}{\alpha_1}+\frac{H_2}{\alpha_2}-1\geqslant0$, then the analysis of the Legendre spectrum is divided into five subcases to examine the minimum value of $l$, as shown below. The Legendre spectrum is the minimum of the five cases.
\begin{figure}[H]
\centering
\begin{tikzpicture}[scale=0.5]
 \fill[red!40] (-3,-3) rectangle (4,3);
\fill[green!40] (4,-3) rectangle (8,3);
\fill[green!30] (4,3) -- (8,3)--(8,6);
\fill[blue!40] (1,3) rectangle (-3,6);
\fill[blue!40] (1,3) rectangle (4,6);
\fill[blue!30] (4,6) -- (4,3)--(8,6);
\draw[thick,dashed] (-3,3) -- (8,3);
\draw[thick,dashed] (4,-3) -- (4,6);
\draw[thick,dashed] (4,3) -- (8,6);
\node [red] at(0.9,0.2){case $4.1$};  
\node [red] at(6,0.2){case $4.2$}; 
\node [red] at(0.9,4.5){case $4.3$};  
\node [red] at(5.5,5){case $4.4$};
\node [red] at(6.5,4){case $4.5$};
\draw [-stealth](-3,0) -- (8,0)node[below,xshift=0.8em,yshift=0.5em]{$r_1$};
\draw [-stealth](1,-3) -- (1,6)node[below,xshift=0em,yshift=1em]{$r_2$};
\node [black] at(3.8,-0.4){$\frac{1}{\alpha_1}$};
\node [black] at(0.6,2.8){$\frac{1}{\alpha_2}$};
\end{tikzpicture}
\caption{}
\end{figure}
\begin{itemize}
    \item If $1-\alpha_1r_1\geqslant 0$, and $1-\alpha_2r_2\geqslant 0$, then based on the assumption of $r_1,r_2,H_1$ and $H_2$, the function $l(r,h)$ is decreasing as $r_1$  and $r_2$ increase. Hence, taking the maximum admissible values $r_1=\frac{1}{\alpha_1},$ and $r_2=\frac{1}{\alpha_2},$ yields $l(r,h)=\frac{H_1}{\alpha_1}+\frac{H_2}{\alpha_2}-1$.

    \item If $1-\alpha_1r_1< 0$, and $1-\alpha_2r_2\geqslant 0$, then $l(r,h)$ is decreasing as $r_1$ decreases and $r_2$ increases. Thereafter, minimizing with respect to $r_1=\frac{1}{\alpha_1},$ and $r_2=\frac{1}{\alpha_2},$ it follows that $l(r,h)=\frac{H_1}{\alpha_1}+\frac{H_2}{\alpha_2}-1$.

    \item If $1-\alpha_1r_1\geqslant 0$, and $1-\alpha_2r_2< 0$, then $l$ is decreasing as $r_1$ increases and $r_2$ decreases. Choosing the same value $r_1=\frac{1}{\alpha_1},$ and $r_2=\frac{1}{\alpha_2},$ again leads to
    $l(r,h)=\frac{H_1}{\alpha_1}+\frac{H_2}{\alpha_2}-1$.

    \item If $1-\alpha_1r_1< 0$, $1-\alpha_2r_2< 0$, and $\alpha_1 r_1\leqslant \alpha_2 r_2$, that is, if $ r_1\leqslant \frac{\alpha_2}{\alpha_1}r_2$, and $\frac{\alpha_1}{\alpha_2}\cdot H_2\geqslant -(H_1-\alpha_1)$, then the minimum value is attained on the line $\alpha_1r_1=\alpha_2r_2$. As a consequence,
\begin{equation}
    \nonumber
   l(r,h)=r_1(H_1-\alpha_1)+r_2H_2= r_1(H_1-\alpha_1)+\frac{\alpha_1}{\alpha_2}r_1H_2=r_1\left(H_1-\alpha_1+\frac{\alpha_1}{\alpha_2}H_2\right).
\end{equation}
Combining the assumption that $\frac{H_1}{\alpha_1}+\frac{H_2}{\alpha_2}-1\geqslant0$, the value of $l(r,h)$ attains  minimum by letting $r_1=\frac{1}{\alpha_1}$,
\begin{equation}
    \nonumber
   l(r,h)= \frac{H_1}{\alpha_1}+\frac{H_2}{\alpha_2}-1.
\end{equation}

    \item If $1-\alpha_1r_1< 0$, $1-\alpha_2r_2< 0$, and $\alpha_1 r_1 >  \alpha_2 r_2$, that is, if $ r_1\geqslant \frac{\alpha_2}{\alpha_1}r_2$, and $H_1 \geqslant -\frac{\alpha_1}{\alpha_2}(H_2-\alpha_2)$, then  the minimum value is reached on the line $\alpha_1r_1=\alpha_2r_2$. As a result,
\begin{equation}
    \nonumber
   l(r,h)=\frac{\alpha_2}{\alpha_1}r_2H_1+r_2(H_2-\alpha_2)=r_2\left[\frac{\alpha_2}{\alpha_1}H_1+(H_2-\alpha_2)\right].
\end{equation}
By $\frac{H_1}{\alpha_1}+\frac{H_2}{\alpha_2}-1\geqslant0$,  the minimum value is attained when $r_2=\frac{1}{\alpha_2}$,
\begin{equation}
    \nonumber
   l(r,h)= \frac{H_1}{\alpha_1}+\frac{H_2}{\alpha_2}-1.
\end{equation}
\end{itemize}
Across all five subcases, if $(H_1,H_2) \in [0,\alpha_1]\times [0,\alpha_2]$, and $\frac{H_1}{\alpha_1}+\frac{H_2}{\alpha_2}-1\geqslant0$, then the Legendre spectrum is
\begin{equation}\label{fbivariatelegendrespectrem1}
\mathcal{L}_{L_{\alpha_1}^b,L_{\alpha_2}^{b,y}}(H_1, H_2)= \min\limits_{r}l(r,h)= \frac{H_1}{\alpha_1}+\frac{H_2}{\alpha_2}-1.
\end{equation}
\end{itemize}
Thus, based on \eqref{fbivariatelegendrespectrem2},\ \eqref{fbivariatelegendrespectrem3}, \eqref{fbivariatelegendrespectrem4}, and \eqref{fbivariatelegendrespectrem1}, when $\{w_k\}$ is a bounded sequence, the multivariate multifractal Legendre spectrum is
\begin{equation} \nonumber
\mathcal{L}_{L_{\alpha_1}^b,L_{\alpha_2}^{b,y}}(H_1, H_2)= \left\{
             \begin{array}{ll}
           \frac{H_1}{\alpha_1}+\frac{H_2}{\alpha_2}-1,&\quad \text{if}\ H\ \in\ [0,\alpha_1]\times [0,\alpha_2],\ \text{and}\ \frac{H_1}{\alpha_1}+\frac{H_2}{\alpha_2}-1\geqslant 0,\vspace{1ex}\\
            -\infty,&\quad \text{else}.
             \end{array}
\right.
\end{equation}

In the following, we analyze the bivariate multifractal Legendre spectrum of $L_{\alpha_1}^b$ and $L_{\alpha_2}^{b,y}$ under the condition that the sequence  $\{w_k\}$ is unbounded. In this case, the function $l(r,h)$ takes the following piecewise form depending on the region of $r$,
\begin{small}
    \begin{equation} \nonumber
    \begin{aligned}
      l(r,H)=
      \left\{
             \begin{array}{ll}
            1+r_1(H_1-\alpha_1)+r_2(H_2-\alpha_2),&\quad \text{if}\ 1-\alpha_1 r_1 \geqslant 0,\ 1-\alpha_2 r_2 \geqslant 0\ \text{and}\  1-\alpha_1 r_1-\alpha_2 r_2 \geqslant 0,\vspace{1ex}\\
            r_1(H_1-\alpha_1)+r_2H_2,&\quad \text{if}\ 1-\alpha_2 r_2 < 0\ \text{and}\  r_1< 0,\vspace{1ex}\\
            r_1H_1+r_2(H_2-\alpha_2),&\quad \text{if}\ 1-\alpha_1 r_1 < 0\ \text{and}\ r_2< 0,\vspace{1ex}
            \\r_1H_1+r_2H_2,&\quad \text{if}\ 1-\alpha_1 r_1 -\alpha_2 r_2 < 0, \ r_1\geqslant 0\ \text{and}\ r_2 \geqslant 0.
             \end{array}
\right.
    \end{aligned}
\end{equation} 
\end{small}

Then the Legendre spectrum is analyzed as following.
\begin{itemize}
    \item If $H_1 > \alpha_1$,\  or $H_2>\alpha_2$, then the term $H\cdot r$ dominates and diverges to $-\infty$ as  $r_1,r_2\to-\infty$. Thereafter,
\begin{equation} \label{bivariatelegendrespectrem2fin}
\mathcal{L}_{L_{\alpha_1}^b,L_{\alpha_2}^{b,y}}(H_1, H_2)= \min\limits_{r}l(r,h)=-\infty.
\end{equation}
    \item 
If $H_1 <0$, or $H_2<0$, then  $l(r,H)\to-\infty$ as $r_1, r_2\to\infty$. Consequently,
\begin{equation} \label{bivariatelegendrespectrem3fin}
\mathcal{L}_{L_{\alpha_1}^b,L_{\alpha_2}^{b,y}}(H_1, H_2)= \min\limits_{r}l(r,h)=-\infty.
\end{equation}
    \item If $(H_1,H_2) \in [0,\alpha_1]\times [0,\alpha_2]$, then we analyse the monotonicity of the function $l(r,h)$ in four parts.   When $1-\alpha_1 r_1 \geqslant 0,\ 1-\alpha_2 r_2 \geqslant 0\ \text{and}\  1-\alpha_1 r_1-\alpha_2 r_2 \geqslant 0$, the function $l(r,h)$ is decreasing as $r_1$ and $r_2$ increase.  When $1-\alpha_1 r_1 < 0\ \text{and}\ r_2 < 0$, the function $l(r,h)$ is decreasing as $r_1$ decreases and $r_2$ increases. When $1-\alpha_2 r_2 < 0\ \text{and}\ r_1 < 0$, the function $l(r,h)$ is decreasing as $r_1$ increases and $r_2$ decreases. When $r_1 \geqslant 0,\ r_2 \geqslant 0\ \text{and}\  1-\alpha_1 r_1-\alpha_2 r_2 < 0$, the function $l(r,h)$ is decreasing as $r_1$ and $r_2$ decrease. It follows that $l(r,h)$ is a convex function and reaches  its minimum at $1-\alpha_1 r_1-\alpha_2 r_2 = 0$, and $0\leqslant r_1\leqslant \frac{1}{\alpha_1}$. Hence, 
\begin{equation}
    \begin{aligned}
        \nonumber
l(r,h)=r_1H_1+\frac{1-\alpha_1r_1}{\alpha_2}H_2=r_1\left(H_1-\frac{\alpha_1}{\alpha_2}H_2\right)+\frac{H_2}{\alpha_2}.
    \end{aligned}
\end{equation}
When $H_1\geqslant \frac{\alpha_1}{\alpha_2}H_2$,  taking $r_1=0$ minimizes the expression, yielding
\begin{equation}
    \nonumber
    l(r,h)=\frac{H_2}{\alpha_2}.
\end{equation}
When $H_1< \frac{\alpha_1}{\alpha_2}H_2$, choosing $r_1=\frac{1}{\alpha_1}$ gives
\begin{equation}
    \nonumber
    l(r,h)=\frac{H_1}{\alpha_1}.
\end{equation}

Thus,
\begin{equation} \label{bivariatelegendrespectrem4fin}
\mathcal{L}_{L_{\alpha_1}^b,L_{\alpha_2}^{b,y}}(H_1, H_2)= \min\limits_{r}l(r,h)=\min\left\{\frac{H_1}{\alpha_1},\frac{H_2}{\alpha_2} \right\}.
\end{equation}
\end{itemize}
To summarize, when $\{w_k\}$ is an unbounded sequence, then the multivariate multifractal Legendre spectrum is given by
\begin{equation}
        \nonumber
\mathcal{L}_{L_{\alpha_1}^b,L_{\alpha_2}^{b,y}}(H_1, H_2)= \begin{cases}
    \min\left\{\frac{H_1}{\alpha_1}, \frac{H_2}{\alpha_2}\right\},       & \quad {\rm if}\ (H_1,H_2) \in \left[0,\alpha_1\right]\times\left[0,{\alpha_2}\right],\\
-\infty       &\quad {\rm else}.
  \end{cases}
    \end{equation}



\end{proof}

\subsection{Proof of Theorem \ref{multivariatemultifractalLegendrespectrum_nfunctionsrational}} \label{section5.2}
In this subsection, we prove the multivariate multifractal spectrum for $n$ L\'evy functions $L_{\alpha_i}^{b,-y_i}$ ($1\leqslant i\leqslant n$), under the setting where $y_1=0$, and for all $(2\leqslant i \leqslant n)$, the shift $y_i$ belongs to the class $\mathcal{S}_3,$, as defined in \eqref{s123}. 
\begin{proof}
Let $j$ denote the scale of the $b$-adic interval $\lambda$. By Lemmas \ref{bi_oscillation_rational} and \ref{lemma_structurefunc_rational}, as in the proof of Theorem \ref{multitheofinal}, the oscillations $d_{\lambda}^{(i)}(1\leqslant i\leqslant n)$ can be approximated as $d_{\lambda}^{(i)}\sim\Delta\left(\frac{k_2}{b^{l}}\right),\ \forall 1\leqslant l\leqslant j.$ The structure function $S(r,j)$, as defined in \eqref{sructurefunction}, becomes
 \begin{equation}
    \begin{aligned}
        \nonumber
        S(r,j)=b^{-j}\sum\limits_{\lambda\in\Lambda_j}\prod\limits_{i=1}^n\left(d_{\lambda}^{(i)}\right)^{r_i}
        \sim b^{-j}\left(\sum\limits_{l=0}^j b^l b^{-l\sum\limits_{i=1}^n\alpha_ir_i}\right).
    \end{aligned} 
\end{equation}
We now analyze the scaling function $\zeta(r)$, as defined by the logarithmic slope in \eqref{scalingfunction}. 
\begin{itemize}
    \item If $1-\sum\limits_{i=1}^n\alpha_ir_i \geqslant 0,$ then the sum behaves as a geometric progression dominated by the largest term, yielding
    \begin{equation}
        \nonumber
        \sum\limits_{l=0}^j b^l b^{-l\sum\limits_{i=1}^n\alpha_ir_i}\sim b^{j(1- \sum\limits_{i=1}^n\alpha_ir_i)}.
    \end{equation}
The derivation of the scaling function follows : 
\begin{equation}
\zeta(r)=\liminf\limits_{j\to+\infty}\frac{\log(S(r,j))}{\log(b^{-j})}=\liminf\limits_{j\to+\infty}\frac{\log\left(b^{-j+j(1- \sum\limits_{i=1}^n\alpha_ir_i)}  \right)}{\log(b^{-j})}
      =\sum\limits_{i=1}^n\alpha_ir_i.
\end{equation}
    \item If $1-\sum\limits_{i=1}^n\alpha_ir_i< 0$, then the exponent in the summation becomes negative, and the series converges to a constant as $j\to \infty$. That is
    \begin{equation}
        \nonumber
        \sum\limits_{l=0}^{j}b^{l(1- \sum\limits_{i=1}^n\alpha_ir_i)}\sim 1,
    \end{equation}
and
\begin{equation}\nonumber
\begin{aligned}
\zeta(r)=\liminf\limits_{j\to+\infty}\frac{\log(S(r,j))}{\log(b^{-j})}
=\liminf\limits_{j\to+\infty}\frac{\log\left(b^{-j}  \right)}{\log(b^{-j})}=1.
\end{aligned}
\end{equation}
\end{itemize}
It follows that   the scaling function is given by
\begin{equation} \nonumber
    \begin{aligned}
      \zeta(r)=
      \left\{
             \begin{array}{ll}
            \sum\limits_{i=1}^n\alpha_ir_i,&\quad \text{if}\ 1-\sum\limits_{i=1}^n\alpha_ir_i\geqslant 0,\\
            1,&\quad \text{if}\ 1-\sum\limits_{i=1}^n\alpha_lr_l < 0.
             \end{array}
\right.
    \end{aligned}
\end{equation}
Considering the function $l(r,H)=1-\zeta(r)+H\cdot r$ as defined in \eqref{l(r,h)}, it holds that
\begin{equation}
    \begin{aligned}
        \nonumber
        l(r,H)=\left\{
             \begin{array}{ll}
            1+\sum\limits_{i=1}^nr_i(H_i-\alpha_i),&\quad \text{if}\ 1-\sum\limits_{i=1}^n\alpha_ir_i \geqslant 0,\\
            \sum\limits_{i=1}^n r_iH_i,&\quad \text{if}\ 1-\sum\limits_{i=1}^n\alpha_ir_i < 0.
             \end{array}
\right.
    \end{aligned}
\end{equation}
Using the same reasoning as in the bivariate case (\eqref{bivariatelegendrespectrem141y}, \eqref{bivariatelegendrespectrem142y}, \eqref{bivariatelegendrespectrem1432y}, \eqref{bivariatelegendrespectrem143y}, and \eqref{bivariatelegendrespectrem144y}), the following conclusion holds:
\begin{itemize}
    \item if $(H_1,H_2,\cdots,H_n)\notin [0,\alpha_1]\times[0,\alpha_2]\times\cdots\times[0,\alpha_n]$, then $\min\limits_{r}l(r,h)=-\infty;$
    \item if $(H_1,H_2,\cdots,H_n)\in [0,\alpha_1]\times[0,\alpha_2]\times\cdots\times[0,\alpha_n]$, then the minimum is attained on the hyperplane $1-\sum\limits_{i=1}^n\alpha_ir_i=0$; that is, 
    \begin{equation}
        \begin{aligned}
            \nonumber
           l(r,h)=\sum\limits_{i=1}^n r_iH_i
            =\frac{1-\sum\limits_{i=2}^n r_i\alpha_i}{\alpha_1}H_1+\sum\limits_{i=2}^n r_iH_i
            =\frac{H_1}{\alpha_1}+\sum\limits_{i=2}^n r_i\left(H_i-\frac{\alpha_i}{\alpha_1}H_1\right).
        \end{aligned}
    \end{equation}
    Hence, if there exists any $i(2\leqslant i \leqslant n)$ such that $H_i\neq \frac{\alpha_i}{\alpha_1}H_1$,  then the second term in the expression diverges to $-\infty$, implying $\min\limits_{r}l(r,h)$. Otherwise, if $H_i\neq \frac{\alpha_i}{\alpha_1}H_1$ for all $2\leqslant i \leqslant n$, then the minimum of $l(r,h)$ is $\frac{H_1}{\alpha_1}$.
\end{itemize}
\end{proof}

\subsection{Proof of Theorem \ref{multivariatemultifractalLegendrespectrum_nfunctions}} \label{section5.3}
In this subsection, we determine the multivariate multifractal Legendre spectrum of $n$ functions $L_{\alpha_i}^{b,-y_i}$ ($1\leqslant i\leqslant n$) where $y_1=0, y_i\in \mathcal{S}_3$ for $2\leqslant i\leqslant m,$ and $y_i\in \mathcal{S}_1$ for $m+1\leqslant i \leqslant n$. Here, the set $\mathcal{S}_1$ and $\mathcal{S}_3$ are defined in \eqref{s123}.

\begin{proof}
By Lemmas \ref{bi_oscillation_irrational} and \ref{scalingfunctionirrational}, following the approach of Theorem \ref{multitheofinal}, the oscillations $d_{\lambda}^{(i)}(1\leqslant i\leqslant n)$ can be approximated as follows. If $l<j$, then
\begin{equation} \nonumber
  \begin{aligned}
      &d_{\lambda}^{(1)}\sim\Delta\left(\frac{k_1}{b^{l}}\right), 
    \\& \forall 2\leqslant i\leqslant m,\ d_{\lambda}^{(i)}\sim\Delta\left(\frac{k_2}{b^{l}}\right),
    \\&\forall m+1\leqslant i\leqslant n,\  d_{\lambda}^{(i)}\sim\Delta\left(\frac{k_2}{b^{j}}\right);
  \end{aligned}
\end{equation}
if $l=j$, then
\begin{equation} \nonumber
  \begin{aligned}
      &d_{\lambda}^{(1)}\sim\Delta\left(\frac{k_1}{b^{j}}\right), 
     \\& \forall 2\leqslant i\leqslant m,\ d_{\lambda}^{(i)}\sim\Delta\left(\frac{k_2}{b^{j}}\right),
    \\& \forall m+1\leqslant i\leqslant n,\ \forall 1\leqslant l'\leqslant j\ d_{\lambda}^{(i)}\sim\Delta\left(\frac{k_2}{b^{l'}}\right).
  \end{aligned}
\end{equation}
The multivariate multifractal structure function $S(r,j)$, as introduced in \eqref{sructurefunction}, is given as follows. Its asymptotic behavior is governed by contributions from terms associated with different ranges of $i$, leading to
\begin{equation}
    \begin{aligned}
        \nonumber
        S(r,j)=b^{-j}\sum\limits_{\lambda\in\Lambda_j}\prod\limits_{i=1}^n\left(d_{\lambda}^{(i)}\right)^{r_i}
        \sim b^{-j}\left(\sum\limits_{l=0}^j b^l b^{-l\sum\limits_{i=1}^m\alpha_ir_i} b^{-j\sum\limits_{i=m+1}^n\alpha_ir_i} + \sum\limits_{l=0}^j b^l b^{-l\sum\limits_{i=m+1}^n\alpha_ir_i} b^{-j\sum\limits_{i=1}^m\alpha_ir_i} \right)
    \end{aligned} 
\end{equation}

Next, the scaling function $\zeta(r)$, defined by \eqref{scalingfunction}, is analyzed as following. 
\begin{itemize}
    \item If $1-\sum\limits_{i=1}^m\alpha_ir_i \geqslant 0,$ and $1-\sum\limits_{i=m+1}^n\alpha_ir_i \geqslant 0,$ then
    \begin{equation}
        \nonumber
        \sum\limits_{l=0}^j b^l b^{-l\sum\limits_{i=1}^m\alpha_ir_i-j\sum\limits_{i=m+1}^n\alpha_ir_i}+\sum\limits_{l=0}^j b^l b^{-l\sum\limits_{i=m+1}^n\alpha_ir_i-j\sum\limits_{i=1}^m\alpha_ir_i}\sim b^{j(1- \sum\limits_{i=1}^n\alpha_ir_i)}.
    \end{equation}
The summation in $S(r,j)$ is dominated by its largest term. Hence, the scaling function becomes
\begin{equation}\nonumber
\begin{aligned}
\zeta(r)=\liminf\limits_{j\to+\infty}\frac{\log(S(r,j))}{\log(b^{-j})}=\liminf\limits_{j\to+\infty}\frac{\log\left(b^{-j+j(1- \sum\limits_{i=1}^n\alpha_ir_i)}  \right)}{\log(b^{-j})}
    =\sum\limits_{i=1}^n\alpha_ir_i.
\end{aligned}
\end{equation}
    \item If $1-\sum\limits_{i=1}^m\alpha_ir_i \geqslant 0,$ and $1-\sum\limits_{i=m+1}^n\alpha_ir_i < 0,$ then
    \begin{equation}
        \nonumber
        \sum\limits_{l=0}^j b^lb^{-l\sum\limits_{i=1}^m\alpha_ir_i-j\sum\limits_{i=m+1}^n\alpha_ir_i}+\sum\limits_{l=0}^j b^l b^{-l\sum\limits_{i=m+1}^n\alpha_ir_i-j\sum\limits_{i=1}^m\alpha_ir_i}\sim b^{j(1- \sum\limits_{i=1}^n\alpha_ir_i)}+  b^{-j \sum\limits_{i=1}^m\alpha_ir_i}. 
    \end{equation}
Since $ 1+\sum\limits_{i=1}^m\alpha_ir_i<\sum\limits_{i=1}^n\alpha_ir_i$, it holds that
\begin{equation}\nonumber
\begin{aligned}
\zeta(r)=\liminf\limits_{j\to+\infty}\frac{\log(S(r,j))}{\log(b^{-j})}=\liminf\limits_{j\to+\infty}\frac{\log\left(b^{-j+j(1- \sum\limits_{i=1}^n\alpha_ir_i)}+ b^{-j-j\sum\limits_{i=1}^m\alpha_ir_i} \right)}{\log(b^{-j})}
 =1+\sum\limits_{i=1}^m\alpha_ir_i.
\end{aligned}
\end{equation}
\item If $1-\sum\limits_{i=1}^m\alpha_ir_i < 0,$ and $1-\sum\limits_{i=m+1}^n\alpha_ir_i \geqslant 0,$ then the structure function is 
    \begin{equation}
        \nonumber
        \sum\limits_{l=0}^j b^l b^{-l\sum\limits_{i=1}^m\alpha_ir_i-j\sum\limits_{i=m+1}^n\alpha_ir_i}+\sum\limits_{l=0}^j b^l b^{-l\sum\limits_{i=m+1}^n\alpha_ir_i-j\sum\limits_{i=1}^m\alpha_ir_i}\sim b^{-j \sum\limits_{i=m+1}^n\alpha_ir_i}+b^{j(1- \sum\limits_{i=1}^n\alpha_ir_i)}. 
    \end{equation}
Because $1+\sum\limits_{i=m+1}^n\alpha_ir_i<\sum\limits_{i=1}^n\alpha_ir_i$, it holds that
\begin{small}
\begin{equation}\nonumber
\begin{aligned}
\zeta(r)=\liminf\limits_{j\to+\infty}\frac{\log(S(r,j))}{\log(b^{-j})}=\liminf\limits_{j\to+\infty}\frac{\log\left( b^{-j-j\sum\limits_{i=m+1}^n\alpha_ir_i}+b^{-j+j(1- \sum\limits_{i=1}^n\alpha_ir_i)} \right)}{\log(b^{-j})}
 =1+\sum\limits_{i=m+1}^n\alpha_ir_i.
\end{aligned}
\end{equation}
\end{small}
\item If $1-\sum\limits_{i=1}^m\alpha_ir_i < 0,$ and $1-\sum\limits_{i=m+1}^n\alpha_ir_i < 0,$ then
    \begin{equation}
        \nonumber
        \sum\limits_{l=0}^j b^l b^{-l\sum\limits_{i=1}^m\alpha_ir_i-j\sum\limits_{i=m+1}^n\alpha_ir_i}+\sum\limits_{l=0}^j b^l b^{-l\sum\limits_{i=m+1}^n\alpha_ir_i-j\sum\limits_{i=1}^m\alpha_ir_i}\sim b^{-j \sum\limits_{i=m+1}^n\alpha_ir_i}+b^{-j \sum\limits_{i=1}^m\alpha_ir_i}. 
    \end{equation}
Consequently, the scaling function is given by
\begin{equation} \nonumber
    \begin{aligned}
      \zeta(r)=
      \left\{
             \begin{array}{ll}
            1+\sum\limits_{i=1}^m\alpha_ir_i,&\quad \text{if}\ \sum\limits_{i=1}^m\alpha_ir_i\leqslant  \sum\limits_{i=m+1}^n\alpha_ir_i,\\
            1+\sum\limits_{i=m+1}^n,&\quad \text{if}\ \sum\limits_{i=1}^m > \sum\limits_{i=m+1}^n.
             \end{array}
\right.
    \end{aligned}
\end{equation}
\end{itemize}
Thus, the scaling function is
\begin{small}
\begin{equation}
    \begin{aligned}
        \nonumber
         \zeta(r)=\liminf\limits_{j\to+\infty}\frac{\log(S(r,j))}{\log(b^{-j})}=
      \left\{
             \begin{array}{ll}
            \sum\limits_{i=1}^n\alpha_ir_i,&\quad \text{if}\ 1-\sum\limits_{i=1}^m\alpha_ir_i \geqslant 0\ \text{and}\  1-\sum\limits_{i=m+1}^n\alpha_ir_i \geqslant 0,\\
            1+\sum\limits_{i=1}^m\alpha_ir_i,&\quad \text{if}\ 1-\sum\limits_{i=m+1}^n\alpha_ir_i < 0\ \text{and}\  \sum\limits_{i=1}^m\alpha_ir_i\leqslant \sum\limits_{i=m+1}^n\alpha_ir_i,
            \\
            1+\sum\limits_{i=m+1}^n\alpha_ir_i,&\quad \text{if}\ 1-\sum\limits_{i=1}^m\alpha_ir_i< 0\ \text{and}\  \sum\limits_{i=1}^m\alpha_ir_i>  \sum\limits_{i=m+1}^n\alpha_ir_i,
             \end{array}
\right.
    \end{aligned}
\end{equation}
\end{small}
Considering the function $l(r,H)$ as defined in \eqref{l(r,h)}, it follows that
\begin{equation}
    \begin{aligned}
        \nonumber
        l(r,H)=\left\{
             \begin{array}{ll}
            1+\sum\limits_{i=1}^nr_i(H_i-\alpha_i),&\quad \text{if}\ 1-\sum\limits_{i=1}^m\alpha_ir_i \geqslant 0\ \text{and}\  1-\sum\limits_{i=m+1}^n\alpha_ir_i \geqslant 0,\\
            \sum\limits_{i=1}^mr_i(H_i-\alpha_i)+\sum\limits_{i=m+1}^n r_iH_i,&\quad \text{if}\ 1-\sum\limits_{i=m+1}^n\alpha_ir_i < 0\ \text{and}\  \sum\limits_{i=1}^m\alpha_ir_i\leqslant \sum\limits_{i=m+1}^n\alpha_ir_i,
            \\
            \sum\limits_{i=1}^m r_iH_i+\sum\limits_{i=m+1}^n r_i(H_i-\alpha_i),&\quad \text{if}\ 1-\sum\limits_{i=1}^m\alpha_ir_i < 0\ \text{and}\  \sum\limits_{i=1}^m\alpha_ir_i>  \sum\limits_{i=m+1}^n\alpha_ir_i,
             \end{array}
\right.
    \end{aligned}
\end{equation}
Taking the same analysis as \eqref{fbivariatelegendrespectrem2},\ \eqref{fbivariatelegendrespectrem3}, \eqref{fbivariatelegendrespectrem4}, and \eqref{fbivariatelegendrespectrem1}, it follows that
\begin{itemize}
    \item if $(H_1,H_2,\cdots,H_n)\notin [0,\alpha_1]\times[0,\alpha_2]\times\cdots\times[0,\alpha_n]$, then $\min\limits_{r}l(r,h)=-\infty$,
    \item if $(H_1,H_2,\cdots,H_n)\in [0,\alpha_1]\times[0,\alpha_2]\times\cdots\times[0,\alpha_n]$, then the minimum is on $1-\sum\limits_{i=m+1}^n\alpha_ir_i=0$ and $\sum\limits_{i=1}^m\alpha_ir_i= \sum\limits_{i=m+1}^n\alpha_ir_i$.
  
     On one hand, it holds that
    \begin{equation}
        \begin{aligned}
            \nonumber
            \min\limits_r l(r,h)&=\sum\limits_{i=1}^n r_iH_i-1
            \\&=\frac{1-\sum\limits_{i=2}^m r_i\alpha_i}{\alpha_1}H_1+\sum\limits_{i=2}^m r_iH_i+\frac{1-\sum\limits_{i=m+2}^n r_i\alpha_i}{\alpha_1}H_{m+1}+\sum\limits_{i=m+2}^n r_iH_i-1
            \\&=\frac{H_1}{\alpha_1}+\frac{H_{m+1}}{\alpha_{m+1}}-1+\sum\limits_{i=2}^m r_i\left(H_i-\frac{\alpha_i}{\alpha_1}H_1\right)+\sum\limits_{i=m+2}^n r_l\left(H_i-\frac{\alpha_i}{\alpha_{m+1}}H_{m+1}\right).
        \end{aligned}
    \end{equation}
    Hence, if $H_i=\frac{\alpha_i}{\alpha_1}H_1$ for all $2\leqslant i \leqslant m$, and $H_i=\frac{\alpha_i}{\alpha_{m+1}}H_{m+1}$ for all $m+2\leqslant i \leqslant n$, then the minimum of $l(r,h)$ is $\frac{H_1}{\alpha_1}+\frac{H_{m+1}}{\alpha_{m+1}}-1$. Otherwise, the minimum of $l(r,h)$ is $-\infty$.  

  On the other hand, the minimum value is attained on "$\sum\limits_{i=1}^m\alpha_ir_i=  \sum\limits_{i=m+1}^n\alpha_ir_i$", then
  \begin{equation}
        \begin{aligned}
            \nonumber
            \min\limits_r l(r,h)&=\sum\limits_{i=1}^m r_iH_i+\sum\limits_{i=m+1}^n r_i(H_i-\alpha_i)
            \\&=\sum\limits_{i=1}^m r_i\frac{\alpha_i}{\alpha_1}H_1+\sum\limits_{i=m+1}^n r_i\left(\frac{\alpha_i}{\alpha_{m+1}}H_{m+1}-\alpha_i\right)
            \\&=\frac{H_1}{\alpha_1}\sum\limits_{i=1}^m r_i\alpha_i+\left(\frac{H_{m+1}}{\alpha_{m+1}}-1\right)\sum\limits_{i=m+1}^n r_i\alpha_i
            \\&=\frac{H_1}{\alpha_1}\sum\limits_{i=m+1}^n r_i\alpha_i+\left(\frac{H_{m+1}}{\alpha_{m+1}}-1\right)\sum\limits_{i=m+1}^n r_i\alpha_i
            \\&=\sum\limits_{i=m+1}^n r_i\alpha_i\left(\frac{H_1}{\alpha_{1}}+\frac{H_{m+1}}{\alpha_{m+1}}-1\right).
        \end{aligned}
    \end{equation}
    Thereafter, if $\frac{H_1}{\alpha_{1}}+\frac{H_{m+1}}{\alpha_{m+1}}-1<0$, then the minimum of $l(r,h)$ is $-\infty$. If $\frac{H_1}{\alpha_{1}}+\frac{H_{m+1}}{\alpha_{m+1}}-1\geqslant 0$, then the minimum of $l(r,h)$ is $\frac{H_1}{\alpha_{1}}+\frac{H_{m+1}}{\alpha_{m+1}}-1$.
\end{itemize}

\end{proof}

\section{Concluding remarks and numerical illustrations}

\label{sec:concl} 

The bivariate Legendre spectrum of a  L\'evy function $L_{\alpha_1}^b$ and the translated L\'evy function $L_{\alpha_2}^{b,y}$ was determined. For the sequence $\{w_k\}$ associated with $y$ (as defined in \eqref{w_k}), it holds that in the case where  $\{w_k\}$ is a bounded sequence, the bivariate Legendre spectrum fails to yield an upper bounded for the bivariate multifractal spectrum; however, when $\{w_k\}$ is an unbounded sequence, the bivariate Legendre spectrum yields an upper bound for the bivariate multifractal spectrum. In this case, the bivariate multifractal formalism holds. This distinction between bounded and unbounded sequences is essential for understanding the limitations and capabilities of the bivariate Legendre spectrum in capturing the multifractal properties of the L\'evy functions. Moreover, for almost every $y$, the bivariate multifractal Legendre spectra fail to establish an upper bound on the bivariate multifractal spectra.

\label{numres}

We already mentioned that, in the case of discontinuous functions,  the Legendre spectrum based on wavelet techniques does not necessarily yield an upper bound for the multifractal spectrum; this motivated our choice of using oscillations instead. However, it is  natural to wonder what  would be the output of these methods.  This motivated the following preliminary  numerical  investigation where we compare Legendre spectra obtained using  oscillations vs.  wavelet leaders (see  \cite{jaffard2021review} and references therein for a description of this method and its numerical implementation). We performed  these simulations of the theoretical outcomes pertaining to the bivariate multifractal Legendre spectra, following \eqref{1/3}, by using the publicly available  Matlab toolbox {\em Wavelet p-Leader and Bootstrap based MultiFractal analysis (PLBMF)}.

Figure \ref{matlab_oscillation} illustrates the numerical results of bivariate multifractal Legendre spectra when $b=2$, and $y=1/2, 1/3, 1/4, 1/5$, based on  oscillations vs. wavelets. One can observe that the Legendre spectra derived from wavelets and oscillations exhibit remarkable similarity. It is tempting to conclude from this  preliminary study that the  range of validity of the wavelet approach probably is larger than one could a priori expect; therefore  an important open problem is to determine if wavelet techniques may  yield an upper bound for the multifractal spectrum  for a subclass of discontinuous functions that would include natural examples such as  L\'evy processes, L\'evy functions, or even  the larger framework supplied by normally convergent  Davenport series, see \cite{jaffard2004davenport} for a description of the multifractal properties of these series. 
\begin{figure}
    \centering
    \includegraphics[width=2in,height=1.4in]{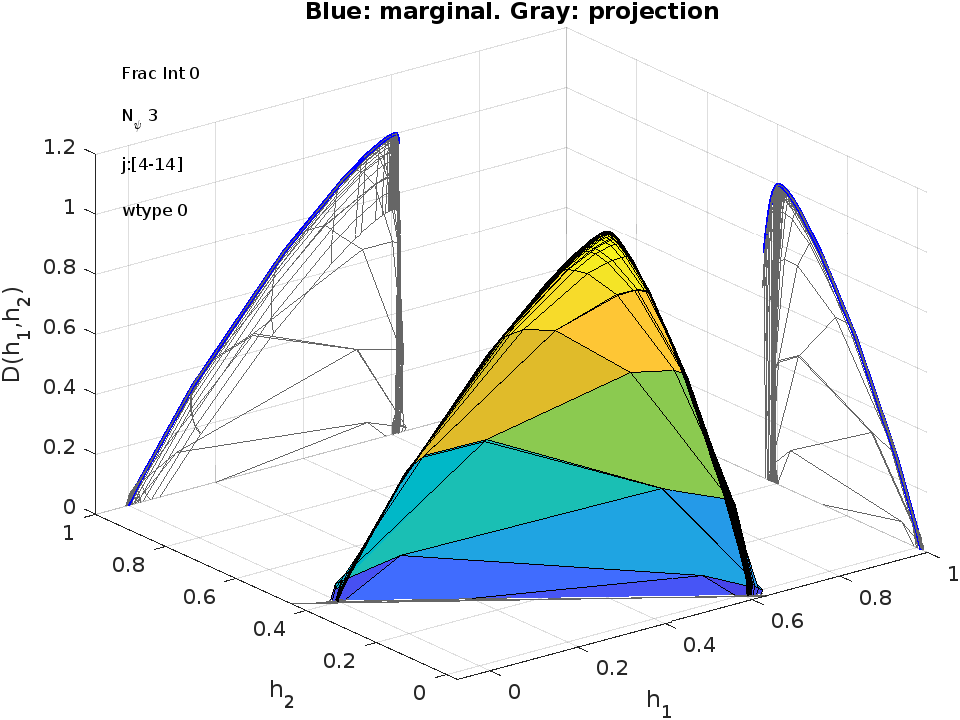}
    \includegraphics[width=2in,height=1.4in]{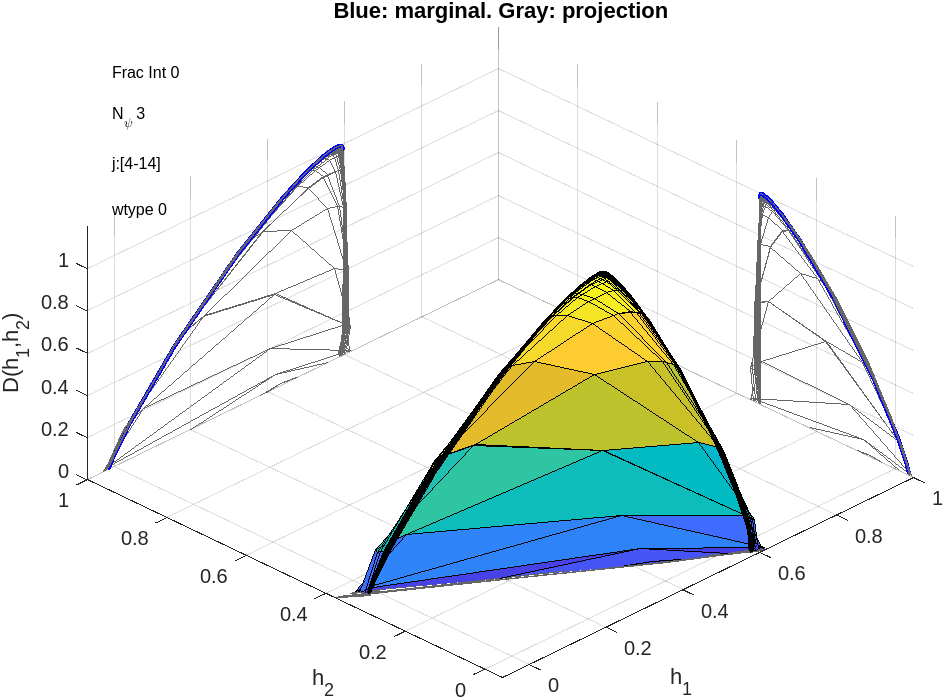}
    \includegraphics[width=2in,height=1.4in]{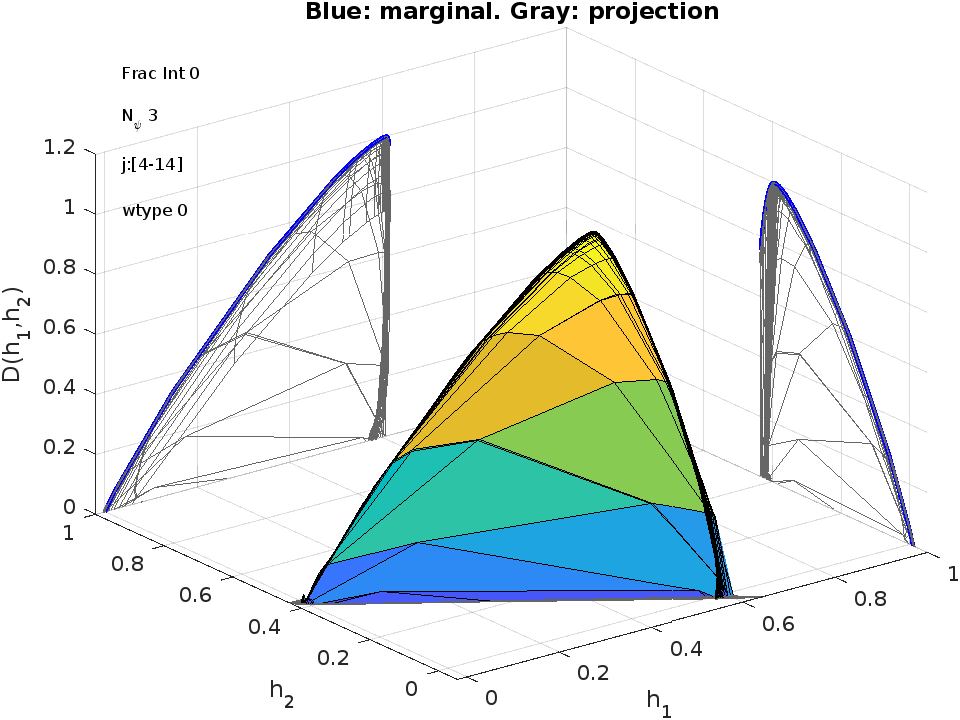}
    \includegraphics[width=2in,height=1.4in]{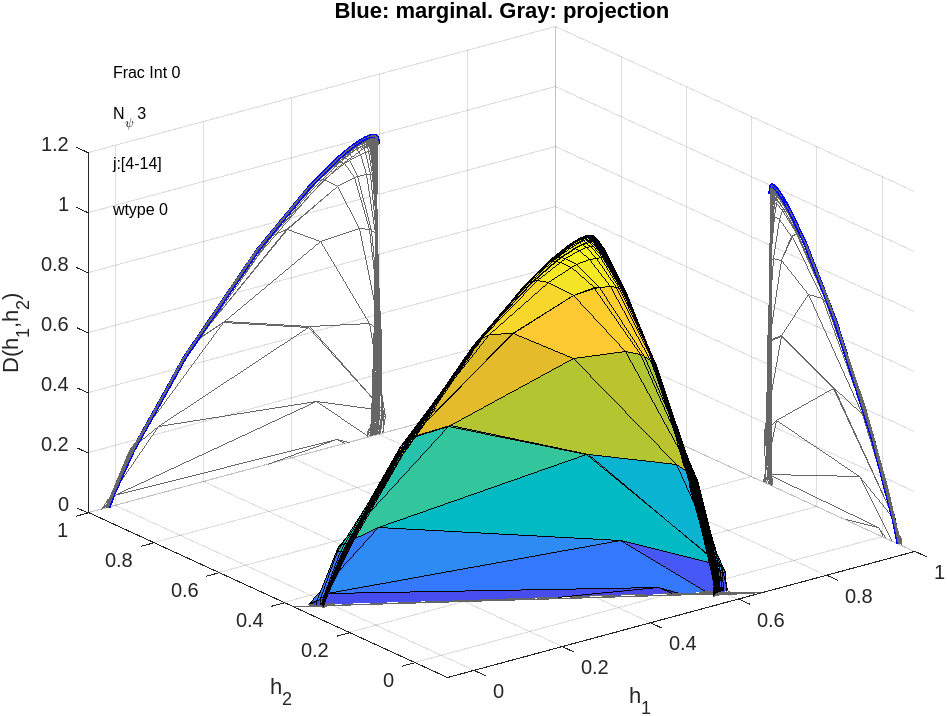}
    \includegraphics[width=2in,height=1.4in]{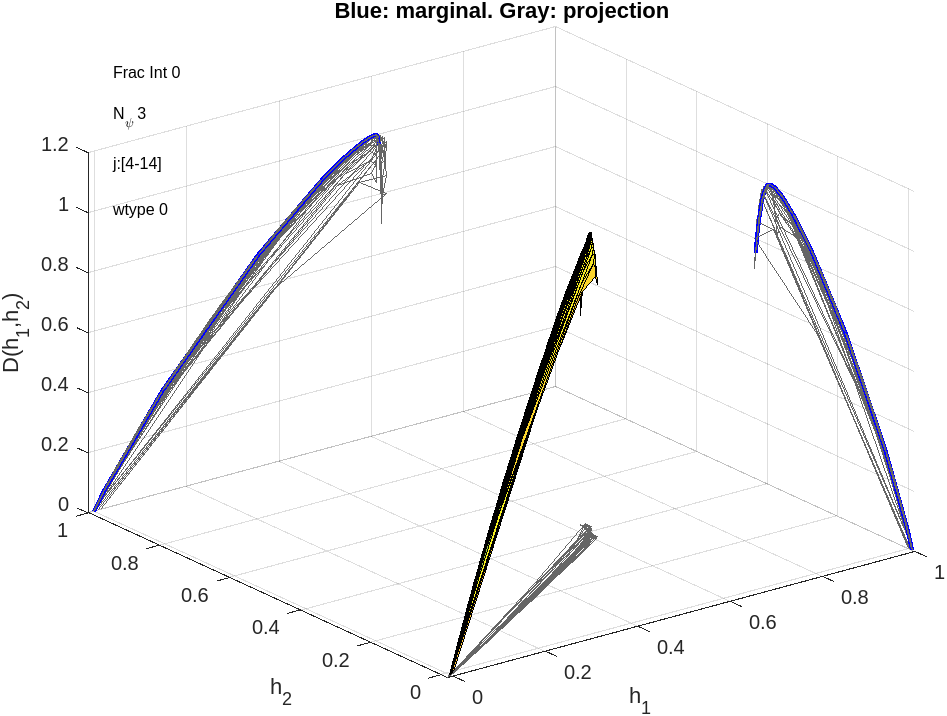}
    \includegraphics[width=2in,height=1.4in]{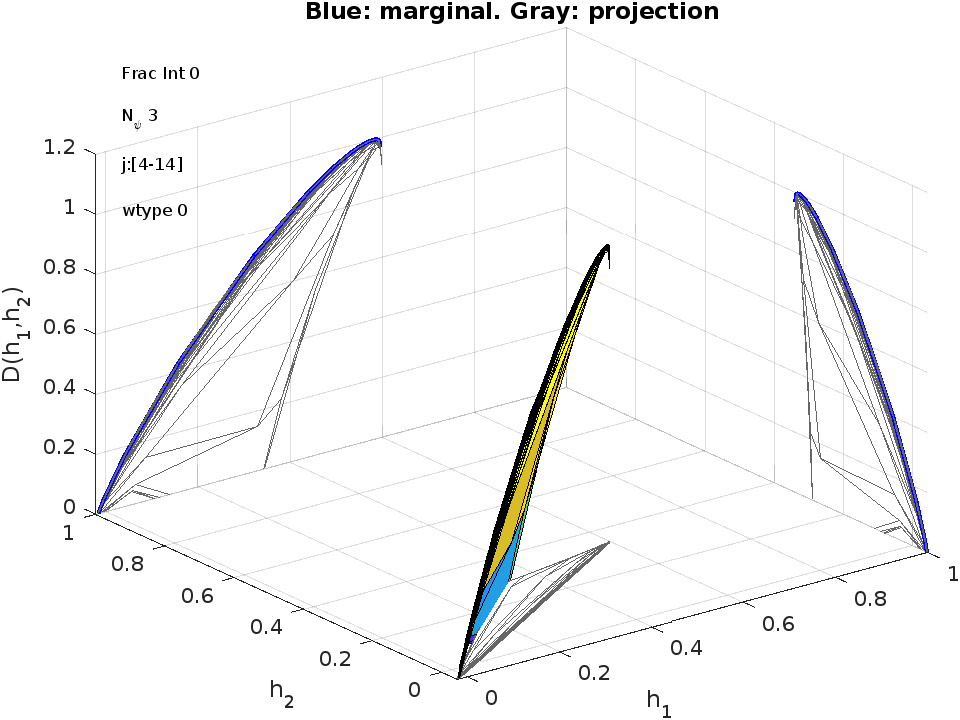}
    \includegraphics[width=2in,height=1.4in]{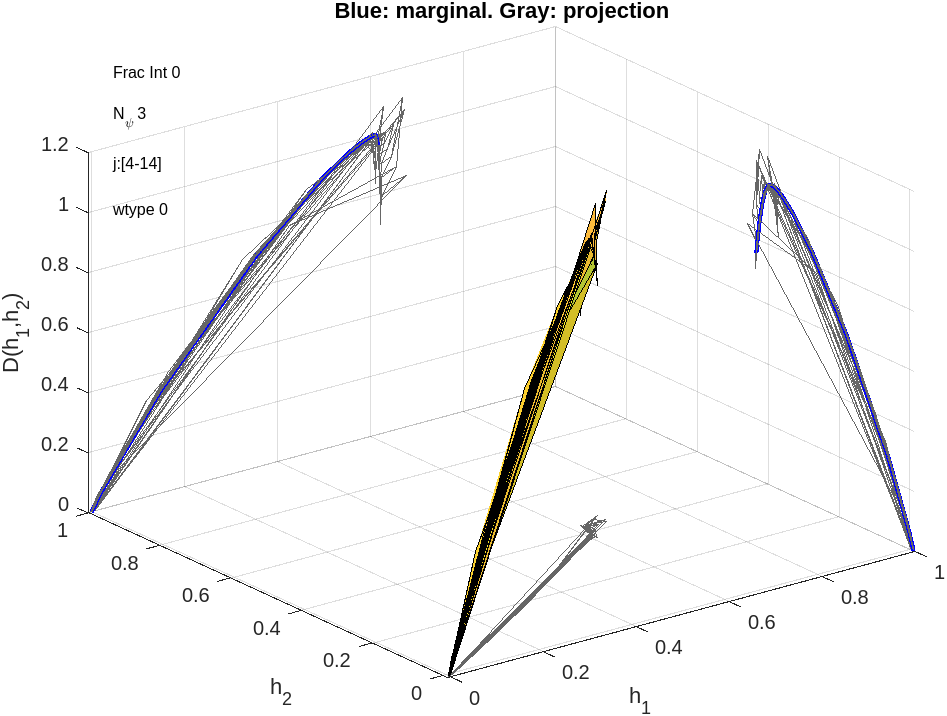}
    \includegraphics[width=2in,height=1.4in]{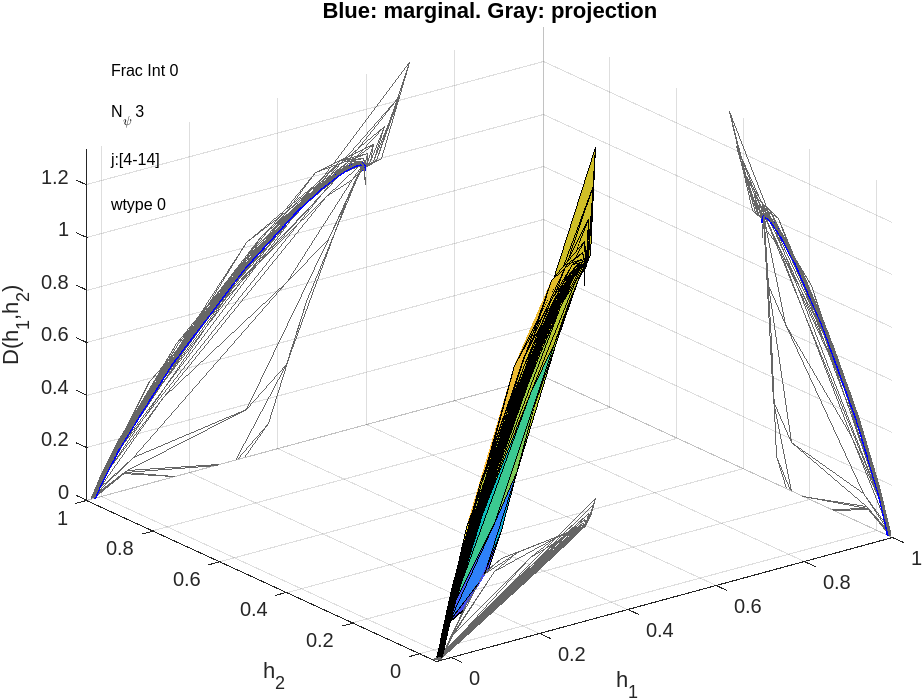}
    \caption{Numerical results of the bivariate multifractal Legendre spectrum $\mathcal{L}_{L_{0.3}^2,L_{0.7}^{2,y}}$ for $y = 1/3, 1/5, 1/2 , 1/4$ based on   oscillations vs.  wavelet leaders ). The left four  bivariate multifractal Legendre spectra are based on order 2 oscillations and the right four are based on wavelet leaders. We also show the projections of the computed bivariate spectra on the planes $H_1 = 0$ and $H_2 =0$  which, theoretically,  yield the corresponding univariate spectra (see \cite{jaffard2019multifractal})}
    \label{matlab_oscillation}
\end{figure}

 Figure \ref{matlab_oscillation1} {\color{blue}(a)(c)} shows the differences between the theoretical and the numerical spectra. It is clear that the numerical results of the bivariate multifractal Legendre spectra, obtained through the oscillations and wavelets, align remarkably well with the anticipated theoretical predictions delineated in equations \eqref{1/3}. This concordance is  depicted in Figure \ref{matlab_oscillation1} {\color{blue}(b)(d)}. Furthermore, it is noteworthy to highlight the error is small between the theoretical expectations and the numerical outcomes. However, it is crucial to acknowledge that the simulation of the "decreasing part" of the Legendre spectra poses a challenge when utilizing the toolbox of Matlab. Consequently, it is not unusual to observe a noticeable marginal error between the theoretical conjectures and the numerical approximations. It must be emphasized that these theoretical results of Legendre spectra as shown in \eqref{1/3} and \eqref{1/2} are derived based on oscillations which, theoretically, cannot be derived if the analysis is based on the wavelets. However, the numerical results of the Legendre spectra can be obtained based on both oscillations and wavelets.
    \begin{figure}[H]
        \centering
            \includegraphics[width=2in,height=1.5in]{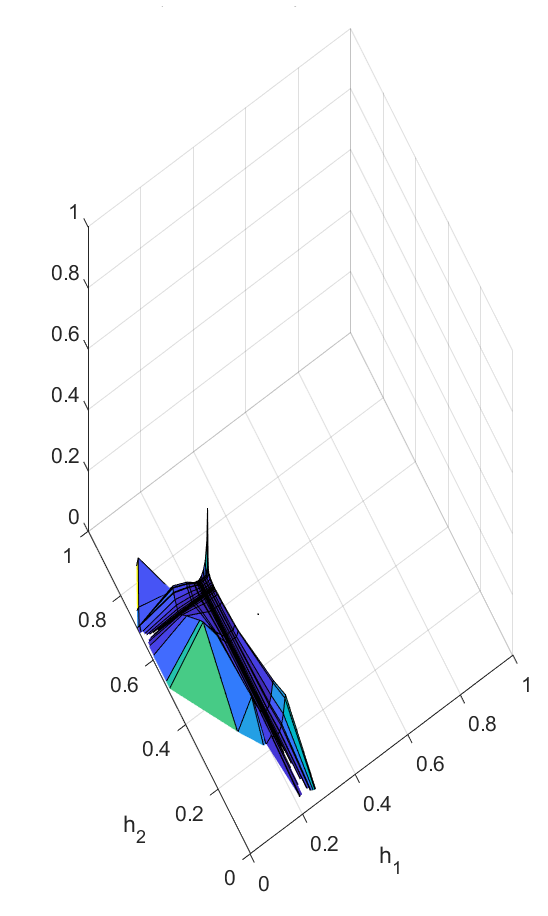}
            \includegraphics[width=2in,height=1.5in]{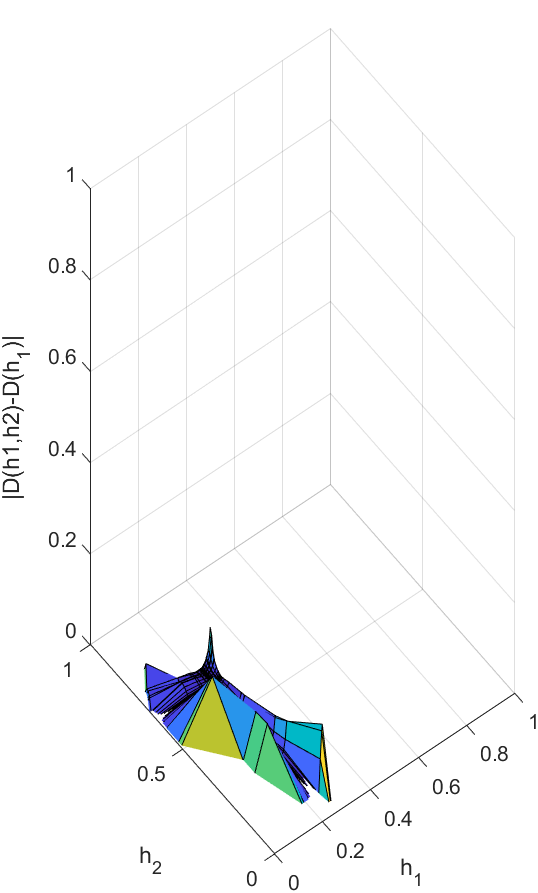}
            \includegraphics[width=2in,height=1.5in]{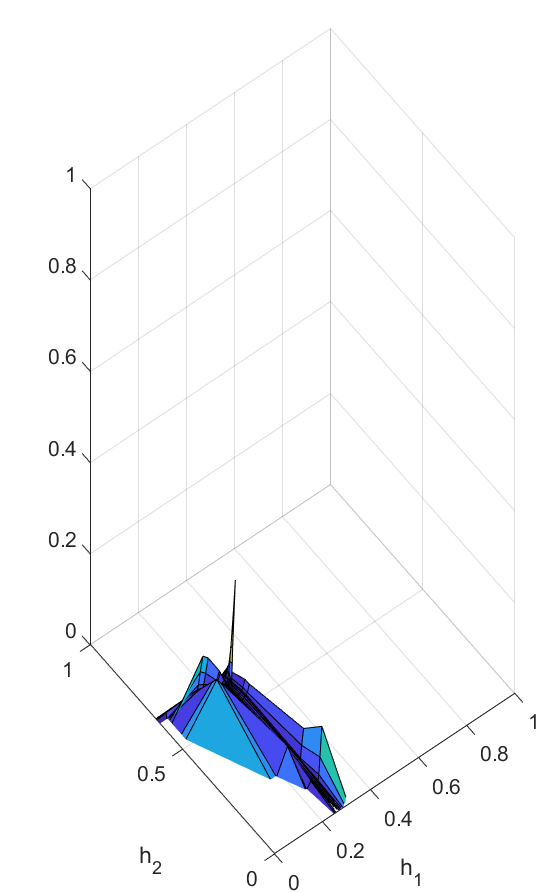}
            \includegraphics[width=2in,height=1.5in]{graph/15_wavelet_difference.png}
        \caption{The left two: the difference between the oscillation Legendre spectrum and the theoretical result based on the oscillations of bivariate Legendre spectra \eqref{1/3}, with $y=1/3$ and $1/5$ respectively. The right two: the difference between the wavelet leader Legendre spectrum and the theoretical result with $y=1/3$ and $1/5$ respectively.}\label{matlab_oscillation1}
    \end{figure}
    
Notably, when $y$ takes values $\frac{1}{3}$ or $\frac{1}{5}$, as depicted in Figure \ref{matlab_oscillation}, the outcomes fit to the "sum of codimension rule" within the region $\frac{H_1}{\alpha_1}+\frac{H_2}{\alpha_2}-1 \geqslant 0$ with $(H_1, H_2) \in [0,\alpha_1]\times [0,\alpha_2]$. The results are in excellent accordance with the theoretical results, as given by \eqref{1/3}, showing the precision of the method we considered. However, it is worth noting the missing part of the Legendre spectrum within the region $\frac{H_1}{\alpha_1}+\frac{H_2}{\alpha_2}-1 < 0$ with $(H_1, H_2) \in [0,\alpha_1]\times [0,\alpha_2]$. The lack of upper bounds for the Legendre spectrum limits its applicability in a way. \\

Our study could be extended in various directions. One of them is to perform the bivariate multifractal analysis of L\'evy functions $L_{\alpha_1}^{b_1}$ and $L_{\alpha_2}^{b_2}$, particularly in the scenario where $b_1$ and $b_2$ are coprime integers. The coprimeness of $b_1$ and $b_2$ introduces additional complexity and richness to the analysis, potentially revealing new insights into the multifractal analysis of these functions. As mentioned by Paul L\'evy himself, the random analogue of L\'evy functions is supplied by L\'evy processes. Therefore another important extension would be to extend the results of the present article to pairs of (independent or shifted) L\'evy processes. 

Another extension concerns  Hecke's function $H_{\alpha}$ which are defined by  
\[ \forall x \in \mathbb{R}, \qquad H_{\alpha}(x)=\frac{\{nx\}}{n^{\alpha}}. \]  They are particular examples of Davenport series, see \cite{jaffard2004davenport}  for which the same questions as those treated in the present two papers can be considered. 

Such studies might give some insight on the key problem of understanding conditions under which the bivariate Legendre spectrum yields an upper bound for the multifractal spectrum.  

Finally, let us mention that Legendre spectra based on oscillations have been considered only in the setting of pointwise regularity supplied by the H\"older exponent. However, it could also be considered in the more general setting of the $p$-exponent, see \cite{jaffard2004wavelet} and references therein for a definition and the multifractal analysis associated with this exponent. In that case, a natural choice for multiresolution quantities would be to consider local $L^p$ norms of finite differences: the { \em $n$-th order $p$-oscillation}  of $f$ on $3 \lambda$ is 
\[ d^{p,n}_\lambda 
=\int\int_{x,\ x+nh \in 3\lambda}|\Delta_f^n(x,h)|^p\; dx \; dh.
 \]
However, the properties of the associated multifractal formalism remain to be explored.

\bibliographystyle{plain}
\bibliography{reference}
\end{document}